\documentclass[12pt]{amsart}

\usepackage{mystyle}

\begin{document}

\title[Well-posedness for a transmission problem]{Well-posedness for a transmission problem connecting first and second-order operators}

\author[H. A. Chang-Lara]{H\'ector A. Chang-Lara}
\address{Department of Mathematics, CIMAT, Guanajuato, Mexico}
\email{hectorchang@cimat.mx}

\begin{abstract}
We establish the existence and uniqueness of viscosity solutions within a domain $\W\subseteq\R^n$ for a class of equations governed by elliptic and eikonal type equations in disjoint regions. Our primary motivation stems from the Hamilton-Jacobi equation that arises in the context of a stochastic optimal control problem.
\end{abstract}

\subjclass{35D40, 35B51, 35F21, 49L12, 49L25}
\keywords{Discontinuous dynamics, transmission problems, viscosity solutions, comparison principle}

\maketitle

\section{Introduction}

Let $\W \subseteq\R^n$ be an open set, and let $\W_B\subseteq\W$ also be open, non-empty, with $\W_E := \W\setminus \overline{\W_B}$ being non-empty as well. We consider a stochastic optimal control problem in which the objective is to minimize the expected time taken by a particle to travel from an initial position $x\in \W$ to its first exit from $\W$. The controller can determine the direction of the particle at any given moment when the particle is in the region $\overline \W_E\cap \W$, maintaining a constant speed, assumed to be one without loss of generality. In the complementary region $\W_B$, the particle follows a Brownian motion. This dynamics allow the particle to switch between the regions multiple times before exiting. See Figure \ref{fig:1}.

\begin{figure}[h]
    \centering
    \includegraphics[width=12cm]{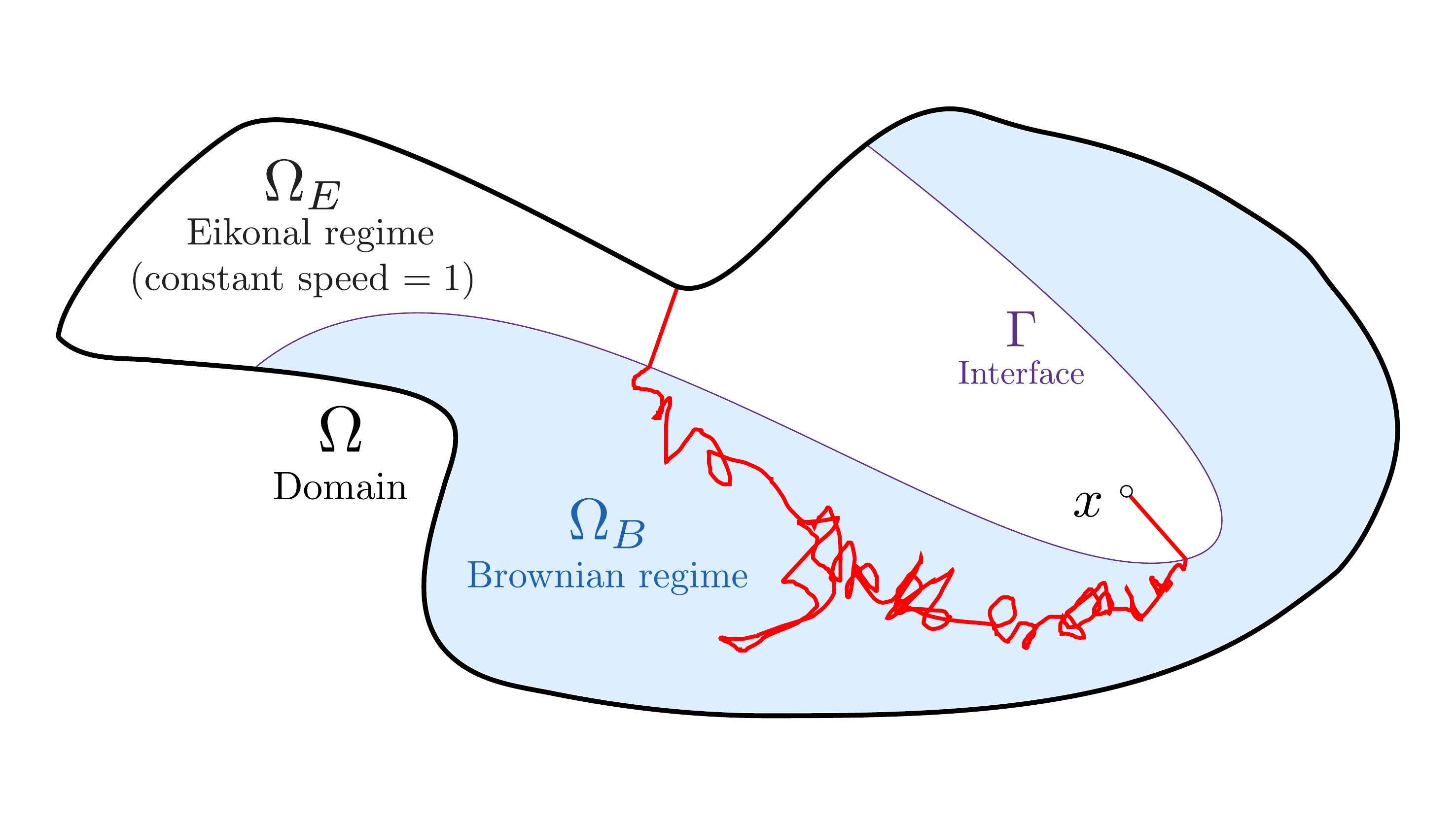}
    \caption{A possible realization of a path from $x$ to $\p\W$. The dynamics on $\W_E$ are chosen by the controller and follow a Brownian motion in $\W_B$.}
    \label{fig:1}
\end{figure}

The analysis of the optimal strategy is closely related to the Hamilton-Jacobi-Bellman equation satisfied by the value function $u:\overline \W\to\R$, defined as the least expected exit time when the particle starts at $x$. In the regions $\W_E$, $\W_B$, and $\partial\W$, the value function satisfies an eikonal equation, a Poisson equation, and a boundary condition, respectively
\begin{align*}
\begin{cases}
|Du| = 1 \text{ in } \W_E,\\
\frac{1}{2}(-\D)u = 1 \text{ in } \W_B,\\
u=0 \text{ in } \p\W
\end{cases}
\end{align*}
In order to obtain a unique solution, we anticipate an additional transmission condition over the interface $\Gamma := \partial \W_E\cap \W = \partial \W_B\cap \W$. The primary goal of this work is to reveal this transmission condition from a partial differential equations point of view. Moreover, we establish the existence and uniqueness of viscosity solutions for a broader family of equations, using the model described above as the principal guiding example.

A discrete implementation of this problem suggests that over the interface $\Gamma$ either one of the equations must hold. Indeed, we have observed that numerical solutions converge to a limit whenever the points in the discretization may include or not points on the interface. When we have an interface point a natural condition would be to consider the two possible dynamics, eikonal or random walk, and choose the one leading to the least expected time. We plan to present a deeper analysis of our findings in a forthcoming work.

The relaxed notion of the problem where either of two given equations can hold over the boundary has been previously used in the fields of optimal control theory and viscosity solutions for quite some time \cite[Section V.4]{MR1484411} and \cite[Section 7]{MR1118699}. It is the natural requirement for achieving stability of viscosity solutions. Specifically, for the sub-solution equations, we must consider
\[
\begin{cases}
|Du| -1 \leq 0 \text{ in } \W_E,\\
\frac{1}{2}(-\D)u -1 \leq 0 \text{ in } \W_B,\\
\min(|Du|-1,\frac{1}{2}(-\D)u-1) \leq 0 \text{ on } \Gamma.
\end{cases}
\]
Meanwhile, for the super-solution equations, we reverse the direction of all inequalities and replace the minimum with a maximum in the expression over the interface
\[
\begin{cases}
|Du| -1\geq 0 \text{ in } \W_E,\\
\frac{1}{2}(-\D)u -1\geq 0 \text{ in } \W_B,\\
\max(|Du|-1,\frac{1}{2}(-\D)u-1) \geq 0 \text{ on } \Gamma.
\end{cases}
\]
A solution is then a function that simultaneously satisfies both sub-solution and super-solution conditions. Colloquially, we say that
\begin{align}\label{eq:main}
\begin{cases}
|Du| -1 = 0 \text{ in } \W_E,\\
\frac{1}{2}(-\D)u -1= 0 \text{ in } \W_B,\\
\text{Either $|Du|-1$ or $\frac{1}{2}(-\D)u -1 = 0$  on $\Gamma$}.
\end{cases}
\end{align}
Although it might appear that the condition over the interface is too weak to guarantee the uniqueness of solutions, our aim is to demonstrate that, in the context of viscosity solutions and flat interfaces, this is surprisingly not the case. We will do this by deducing a stronger equation for the interface and establishing a comparison principle.

There are several challenges associated with the comparison principle for this problem. Firstly, note that the operator governing this equation over $\W$ as a whole is discontinuous across $\Gamma$. This lack of translation invariance renders the comparison principle a non-trivial question. On the other hand, we may hope to apply the comparison principle for general boundary conditions, as established by Barles in \cite{MR1249178}. However, the eikonal operator on $\Gamma$ is not monotone in the exterior normal direction to $\partial\W_B$, which prevents the direct application of such theory.

\subsection{Main results}

In an effort to enhance the applicability of our results to related problems, we have extended our hypotheses beyond the eikonal/Poisson equations illustrated in this introduction. The corresponding hypotheses will be announced in each section. For the moment and to keep things simple, let us state our contributions in the context of the example we have already discussed.

We demonstrate as a consequence of Theorem \ref{thm:strong} that for any $C^2$ regular interface $\Gamma$, the problem \eqref{eq:main} is equivalent to the following stronger equation
\begin{align}\label{eq:main1.5}
\begin{cases}
|Du| -1= 0 \text{ in } \W_E\cup \Gamma,\\
\frac{1}{2}(-\D)u - 1 =0\text{ in } \W_B.
\end{cases}
\end{align}
Colloquially speaking, we could say that the eikonal mechanism dominates the dynamics at the interface. For Theorem \ref{thm:strong} we mainly need a coercivity from the second order operator resulting from uniform ellipticity.

For flat interfaces we establish a comparison principle for the equation \eqref{eq:main1.5} in Corollary \ref{cor:comp}. Our general comparison principle is stated in Theorem \ref{thm:comp} for continuous sub and super-solutions with their corresponding equations being separated by some gap (see the hypotheses \ref{hyp:sub} and \ref{hyp:sup} at the beginning of Section \ref{sec:comparison}). This gap can be taken to be just zero whenever the operators satisfies some further coercivity type assumptions, such as in the case of convex operators (the case in the introduction), parabolic problems, and equations arising from geometric discounted costs, see Corollary \ref{cor:comp} and Remark \ref{rmk:other}. Theorem \ref{thm:comp} also relies on further hypotheses on the operators (see the hypotheses \ref{hyp:hm} and \ref{hyp:hp} also at the beginning of Section \ref{sec:comparison}). 

In the preliminary section we have included, besides the main definitions, a few results which are direct sapplications of the classical theory. For instance, the existence of viscosity solution by Perron's method stated in Theorem \ref{thm:existence}, relies on the global regularity for eikonal type and uniformly elliptic equations.

\subsection{Related work}

One of the main examples of problems that combine disjoint regions with diffusion and eikonal type regimes, emerges in the theory of singular stochastic control, as seen in Chapter VIII of the book by Fleming and Soner \cite{MR2179357}. The representative equation for the value function takes the following form, known as a \textit{gradient constrained problem}
\[
\max(H(Du),Lu)=0 \text{ in }\W.
\]

In the equation above, $L$ is a second-order elliptic operator, and $H$ is typically a convex function. A stochastic control problem that motivates this equation for $Lu=\tfrac{1}{2}(-\D) u-1$ and $H(Du)=|Du|-1$ can be described as follows: We aim to minimize the expected exit time of a particle starting at $x\in \W$, such that at each instant, we can decide to move either with Brownian motion or with speed one in a chosen direction. In a few words, this is an extension of the problem in the introduction where the controlled also can conveniently choose the region $\W_B$.

As usual, the value function $u$ represents the least expected exit time when the particle starts from $x$. From a dynamic programming principle, we derive the Hamilton-Jacobi equation of singular control provided above. This model has significant applications in Merton's portfolio problem and spacecraft control, with multiple references available in \cite[Section VIII.7]{MR2179357}. Other motivations for gradient constrained problems, arising from variational inequalities, can be found in applications of elastoplasticity of materials \cite{MR195316}. The survey \cite{MR3393319} explore the connections with recent developments for the obstacle problem.

The literature concerning the analysis of solutions for gradient constrained problems with a convex $H$ is extensive. Notable works include the $C^{1,1}$ optimal regularity of solutions by Evans \cite{MR544887,MR529814}, Wiegner \cite{MR607553}, Ishii and Koike \cite{MR693645}, and Soner and Shreve \cite{MR1001925,MR1104105}. Brezis and Sibony \cite{MR346345} show the equivalence with an obstacle problem. Recent works include those by Andersson, Shahgholian, and Weiss \cite{MR2989443}, Hynd \cite{MR2898887,MR3023063,MR3621845}, Hynd and Mawi \cite{MR3563780}, and Safdari \cite{MR3353792,MR3667699,MR4200759}. In the last few years, attention has also been given to cases where $H$ is non-convex, motivated by applications in optimal dividend strategies for multiple insurances \cite{MR3982209}. See Safdari \cite{MR4205181} and our own collaboration with Pimentel \cite{MR4249793} for the analysis of solutions.

In the previous scenario, the solution is the one that determines the regions where either the first-order or the second-order operator is active. Furthermore, the solution must globally satisfy $Lu\leq 0$ and $H(Du)\leq 0$, which imposes significant rigidity on the function. The problem considered in this paper fixes the regions and does not assume global relations like in the gradient constrained problem; as a result, solutions are expected to be more flexible across this interface. In our case, we no longer expect better regularity than Lipschitz continuity, this is illustrated by the examples in the Section \ref{sec:exis_one_d} and the Section \ref{sec:visc_sol}.

Another research direction related to our model involves transmission problems. These typically deal with differential equations over given disjoint domains connected by some prescribed condition over their common boundaries. Borsuk's book \cite{MR2676605} provides a detailed exposition of second-order problems. Some recent regularity results have been established by Caffarelli, Soria-Carro, and Stinga in \cite{MR4228861}, and by Soria-Carro and Stinga in \cite{https://doi.org/10.48550/arxiv.2207.13772}. Problems connecting operators with different (fractional) orders have been studied, for instance, by Kriventsov in \cite{MR3356996}, D'Elia, Perego, Bochev, and Littlewood in \cite{MR3501316}, and Capana and Rossi in \cite{MR4534535}. An intermediate step in our proof of Theorem \ref{thm:strong} involves uncovering a transmission-type equation for our problem (Lemma \ref{lemma1}, Corollary \ref{cor:strong_eq}, and Lemma \ref{lemma1.5}).

The works referenced in the previous paragraph feature some ellipticity condition on both sides of the interface. Problems governed by hyperbolic or eikonal-type operators on both sides are also very relevant in applications and have been studied recently. These works are motivated by models in deterministic optimal control. The contributions we have in mind include those by Barles, Briani, and Chasseigne \cite{MR3092359,MR3206981}, Lions and Souganidis \cite{MR3556345,MR3729588}, Imbert and Monneau \cite{MR3690310}, and more recently Barles, Briani, Chasseigne, and Imbert \cite{MR3922443}. Our weak formulation on the interface is analogous to the so-called \textit{natural} junction condition found in \cite[Equations (1.3) and (1.4)]{MR3922443}.

In a recent publication by Imbert and Nguyen \cite{MR3709301}, the authors consider junction problems with second order operators on either side of the interface or junction. The model in this case assumes that the equations degenerate to first order problem at the junction. As in the models discussed in the previous paragraph, there is an additional junction condition from where it is shown that the problem is well posed. Similar to one of our results, an intermediate step consists the equivalence between a relaxed and a strong notion of viscosity solutions (know as flux limited solution).

In contrast to the first order problems, the presence of uniformly elliptic effects up to the interface seems to make our model rigid enough to yield uniqueness of solutions by itself, without additional conditions on the interface which in this case may lead to an over-determined ill-posed problem.

\subsection{Main ideas}

In this brief section we aim to emphasize the fundamental points that underlie the main theorems presented in this paper. We acknowledge that a significant portion of the article is devoted to discussing technical modifications on well-established classical notions in the theory of viscosity solutions, these have been included for the sake of completeness. However, this may obscure the novel ideas and challenges that we have encountered in our work and will appear toward the Section \ref{sec:strong} and the Section \ref{sec:comparison}. To address this concern, we would like to provide a brief discussion in this introduction.

The upcoming presentation is informal and assumes a certain level of familiarity with the concept of viscosity solutions, which will be revisited in Section \ref{sec:preliminaries}.

\subsubsection{Strong equation}

To show that the weak equation \eqref{eq:main} implies the strong equation \eqref{eq:main1.5}, we demonstrate as an intermediate step that \eqref{eq:main} can also be characterized with test functions that may have discontinuous gradients at the interface. In other words, they belong to the following space of functions (here $B_r(x_0)\ss\W$)
\[
C(B_r(x_0))\cap C^2(B_r(x_0)\cap (\W_E\cup\Gamma))\cap C^2(B_r(x_0)\cap (\W_B\cup\Gamma)).
\]
Whenever the contact occurs at the interface, we find that either the eikonal equation holds on the \underline{Brownian} side, or there is a transmission condition between the normal derivatives.

To be more specific, we show that \eqref{eq:main} implies the following problem in an appropriate viscosity formulation
\begin{align*}
\begin{cases}
|Du| -1= 0 \text{ in } \W_E,\\
\frac{1}{2}(-\D)u - 1=0 \text{ in } \W_B,\\
\text{Either $|Du_B|-1$ or $\p_\nu u_E-\p_\nu u_B = 0$  on $\Gamma$}.
\end{cases}
\end{align*}
In the equation above, we have denoted by $\nu$ the interior normal of $\W_B$ over $\Gamma$ (or the exterior normal of $\W_E$ over $\Gamma$). Moreover, $u_B$ and $u_E$ represent the restrictions of $u$ to $\W_B$ and $\W_E$, respectively. Geometrically, $\p_\nu u_E-\p_\nu u_B$ is positive if the graph of $u$ forms a concave angle along $\G$, and negative if the angle is instead convex.

\subsubsection{Existence and uniqueness of solutions in one dimension}\label{sec:exis_one_d}

Let $\W = (-1,1)$, $\W_E = (-1,0)$, $\W_B = (0,1)$, and $\Gamma=\{0\}$. Our goal is to see that there is at most one viscosity solution $u \in C(\overline{\W})$ for 
\begin{align}\label{eq:1dim}
\begin{cases}
    |u'|-1=0 \text{ in } \W_E,\\
    -u''=0 \text{ in } \W_B,\\
    \text{Either $|u'|-1$ or $-u''=0$ on $\Gamma$},\\
    u(-1) = 0,\\
    u(1) = \a.
\end{cases}
\end{align}

The two equations and the condition at $x=-1$ indicate that for some parameter $\b\in[-1,0]$
\[
u(x)=\begin{cases}
1+\b-|x-\b| \text{ in } \W_E,\\
\a x + (1+2\b)(1-x) \text{ in } \W_B.
\end{cases}
\]
See Figure \ref{fig:my_label}. We would like to see that there is at most one $\b$ for which $u$ is a solution.

\definecolor{mycolor}{RGB}{29,99,163}

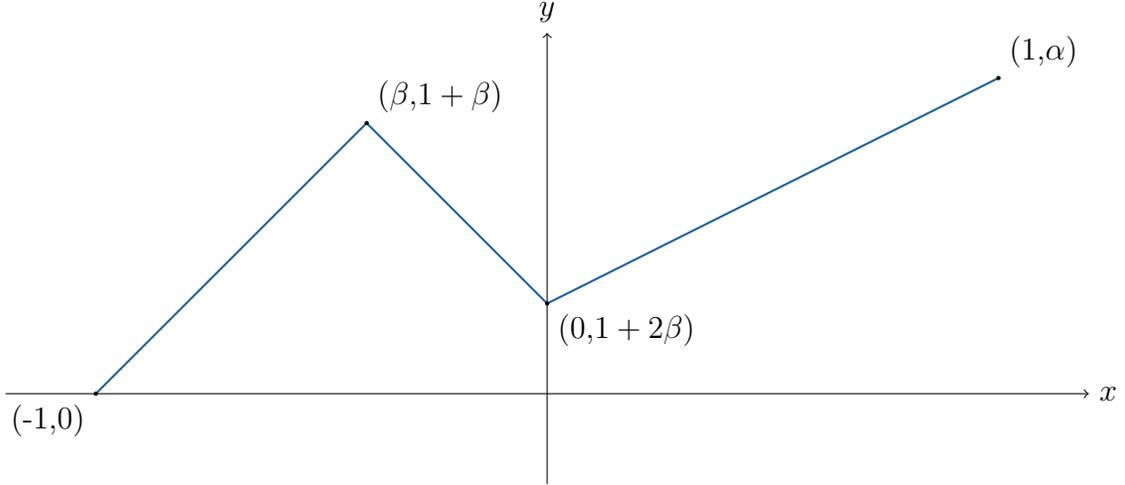
\begin{figure}
    \centering
    \begin{tikzpicture}[scale=6]
    \draw[->] (-1.2,0) -- (1.2,0) node[right] {$x$};
    \draw[->] (0,-0.2) -- (0,0.8) node[above] {$y$};
    
    \draw[thick,mycolor] (-1,0) -- (-0.4,0.6) -- (0,0.2) -- (1,0.7);
    
    \filldraw (-1,0) circle (.1pt) node[below left] {(-1,0)};
    \filldraw (-0.4,0.6) circle (.1pt) node[above right] {($\b$,$1+\b$)};
    \filldraw (0,0.2) circle (.1pt) node[below right] {(0,$1+2\b$)};
    \filldraw (1,0.7) circle (.1pt) node[above right] {(1,$\a$)};
    \end{tikzpicture}
    \caption{Graph of a one dimensional candidate for the solution of \eqref{eq:1dim}.}
    \label{fig:my_label}
\end{figure}

\textbf{Case 1:} $\a \geq 0$. If $\b<0$ then $u$ is not a super-solution because around zero
\[
u(x) = u(0) + \begin{cases}
-x \text{ if } x<0,\\
\gamma x \text{ if } x\geq 0
\end{cases}, \qquad \gamma = \a - (1+2\b) > -1.
\]
Then it can be touched at $x_0=0$ from below by
\[
\varphi(x):=u(0)+\frac{\min(\gamma,1)-1}{2}x + x^2
\]
which does not satisfy $|\varphi'(x_0)|-1\geq 0$ or $-\varphi''(x_0) \geq 0$. The only option left is $\b=0$ in which case
\[
u(x) = \begin{cases}
1+x \text{ if } x<0,\\
(\a-1)x+1 \text{ if } x\geq 0,
\end{cases}
\]
is indeed a viscosity solution. This assertion can be verified in two steps, depending on whether $\alpha$ is larger than $2$ or lies between $0$ and $2$: If $\alpha > 2$, then $u$ can only be touched from below at $x_0 = 0$. In this case, the derivative of the test function is at least $1$, satisfying the super-solution condition. If $\alpha \in [0, 2)$, then $u$ can only be touched from above at $x_0 = 0$. In this case, the derivative of the test function lies between $-1$ and $1$, satisfying the sub-solution condition. If $\alpha = 2$, then $u = 1 + x$, and it clearly satisfies both the sub and super-solution criteria at every point.

\textbf{Case 2:} $\a\in [-2,0)$. Now we can check in a similar fashion as before that $\b=\a/2$ is the only possibility that gives a viscosity solution. This choice of the parameter $\b$ is the one that makes $u$ a translation of the function $-|x|$
\[
u(x) = 1+\b-|x-\b| = 1+\a/2-|x-\a/2|.
\]

\textbf{Case 3:} $\a<-2$. In this case there are no solutions for the given boundary values.

\subsubsection{Comparison principle in one dimension}\label{sec:com_one_d}

Let $\W = (-1,1)$, $\W_E = (-1,0)$, $\W_B = (0,1)$, and $\Gamma=\{0\}$ as in the previous example. Assume that $v$ is a super-solution of
\[
\begin{cases}
    |v'|-1\geq 0 \text{ in } \W_E,\\
    -v''\geq 0 \text{ in } \W_B,\\
   \max(|v'|-1,-v'')\geq 0 \text{ on }\Gamma,
\end{cases}
\]
meanwhile $u$ a sub-solution of the following equation, which leaves a gap $\eta>0$ between both equations
\[
\begin{cases}
    |u'|-1+\eta \leq 0 \text{ in } \W_E,\\
    -u''\leq 0 \text{ in } \W_B,\\
   \min(|u'|-1+\eta,-u'')\leq 0 \text{ on }\Gamma.
\end{cases}
\]

Assuming that both functions are continuous and moreover smooth outside of zero, we will see that a contradiction arises by assuming that the graph of $v$ touches the graph of $u$ from above at the origin.

To the left of zero we find that $u'_E(0) \in [-1+\eta,1-\eta]$. Since $v_E$ is a super-solution touching $u_E$ at zero, we must have that $v_E'(0)\leq -1$.

Let us now examine the behavior of the functions to the right of zero. If $u'_B(0)<-1+\eta$, the graph of $u$ forms a concave angle at zero, with an upper supporting line of slope less than $-1+\eta$. In this case, we can construct a concave paraboloid $\varphi$ that touches $u$ from above at the origin, and further impose $\varphi'(0)<-1+\eta$. However, this contradicts the sub-solution condition at the interface, indicating that $u'_B(0) \geq -1+\eta$ must hold instead. Since $v_B$ touches $u_B$ at zero, we conclude that also $v_B'(0)\geq -1+\eta$. Notice that this step indicates that the eikonal equation gets somehow transmitted towards the Brownian side at the interface.

Putting the two conclusions on $v$ together, we notice that the graph of $v$ forms a convex angle at zero, with a lower supporting plane of slope $-1+\eta/2$. This contradicts $v$ being a super-solution, as it can be touched by a convex paraboloid $\varphi$ with $\varphi'(0)=-1+\eta/2 \in (-1,1)$.

The case with no gap, $\eta=0$, can be recovered by an approximation argument. See the comments after Remark \ref{rmk:other} and Corollary \ref{cor:comp}.

\subsubsection{Challenges and strategies}

The previous and rather simple reasoning in one dimension works because we are assuming some regularity on the solutions. In higher dimensions, the solutions are more flexible because of the variations in the directions parallel to the interface, therefore the comparison principle for viscosity solutions becomes delicate.

Our strategy involves the use of inf/sup regularizations and boundary estimates. We perform inf/sup convolutions along the directions parallel to the flat interface in order to assume that the sub-solution $u$ is semi-convex along the interface, and the super-solution $v$ is semi-concave also along the interface. If we assume by contradiction that $v$ touches $u$ from above at some point $x_0\in \G$, the semi-convexity assumption implies that the functions $C^{1,1}$ at $x_0$ along the interface. In other words, they are squeezed between two paraboloids at $x_0$ and along the interface.

Lemma \ref{lem:44} and its Corollary \ref{cor1} show that if the boundary data for the solution of a uniformly elliptic equation is trapped between two paraboloids, then the solution is differentiable at such point. The main observation is that the normal derivative is well defined. The argument for this proof is due to Caffarelli \cite[Lemma 4.31]{MR787227} and \cite[Theorem 9.31]{Gilbarg-Trudinger2001}. On the other hand, we can see Lemma \ref{lem:H_cons} and Lemma \ref{lem:hopf} as analogous boundary regularity results for solutions of first order equations.

While we can present a fairly general result, our approach does not appear to yield a comparison principle for non-flat interfaces and general non-translation invariant operators.

\subsection{Further questions}

One of the main reasons that justify the analysis of this problem relies on a verification theorem for some particular class of games. In \cite{MR4249793} we proposed and analyzed a Hamilton Jacobi equation for the following one: a particle starts at some position $x\in \W$ and it is driven by two players with opposite goals, the first one wants to maximize the expected exit time from $\W$, meanwhile the second wants to minimize it. The first player chooses $\W_B\ss\W$ and the second player fixes a unit vector field $v:\W_E=\W\sm\overline{\W_B} \to \p B_1$ such that the dynamics are given by the following SDE ($B_t$ is a Brownian motion in $\R^n$)
\[
d\gamma = \mathbbm 1_{\W_E}(\gamma)v(\gamma)dt + \mathbbm 1_{\W_B}(\gamma) dB_t, \qquad \gamma(0)=x.
\]

In a casual manner, we can describe the dynamic of $\gamma$ as being influenced by two players as follows: the first one determines who drives, however when this player takes the wheel it does so in some random fashion and without any preference for any given direction (a drunkard's walk). Meanwhile the second (and sober) player, whenever it gets the opportunity, aims to escape $\W$ by driving at maximum speed in some given set of directions.

The value function for this game is the least expected exit time and should satisfy the gradient constrained problem 
\[
\begin{cases}
\min(|Du|-1,\tfrac{1}{2}(-\D)u-1) = 0 \text{ in } \W,\\
u=0 \text{ on } \p\W.
\end{cases}
\]

The corresponding verification theorem remains an open problem. One possible way to address this question would be to fix the set $\W_B$ and solve a corresponding optimal control problem (for the second player) and then consider the maximum among all the values. If we let $u_{\W_B}$ to be \textit{the solution} of \eqref{eq:main} with zero boundary data, we should then expect that the solution $u$ of the problem above is actually given by the envelope
\[
u = \max_{\W_B\ss\W}u_{\W_B}.
\]
One of the reasons why we are not able to answer this is because \textbf{we do not even know if $u_{\W_B}$ is well defined for a general $\W_B\ss\W$} (or at least a dense set in some suitable topology).

On the other hand, if both players seek to minimize the expected exit time under the same rules, then we are actually considering an optimal control problem where the value function should satisfy the following gradient constrained problem
\[
\begin{cases}
\max(|Du|-1,\tfrac{1}{2}(-\D)u-1) = 0 \in \W,\\
u=0 \text{ on } \p\W.
\end{cases}
\]
This is the simplest model for the problem usually known as gradient constrained that was addressed in the previous section \cite{MR544887,MR529814,MR607553,MR693645,MR1001925,MR1104105,MR346345,MR2989443,MR2898887,MR3023063,MR3621845,MR3563780,MR3353792,MR3667699,MR4200759}. Once again, it would not be surprising that in this case $u=\min_{\W_B\ss\W}u_{\W_B}$. As far as we know, this question has not been addressed before in the literature either.

In either case, the well-posedness for the problems that determine $u_{\W_B}$ not only fulfills a theoretical inquiry. They can also be used to estimate how far from optimal is a given strategy, as they are bounds for the corresponding value function of interest.

\subsection{Organization of the paper}

In Section \ref{sec:preliminaries}, we provide precise definitions of viscosity solutions for equations with discontinuous dynamics. We revisit some classical results, such as stability and Perron's method. The existence of solutions, stated in Theorem \ref{thm:existence}, is a rather straightforward consequence of known results in the classical theory and does not require the comparison principle.

Section \ref{sec:strong} is devoted to show that the equations \eqref{eq:main} and \eqref{eq:main1.5} are equivalent. This result can be interpreted as the canonical transmission condition at the interface. Theorem \ref{thm:strong} is stated to a more general class of problems.

The main result of this work, the comparison principle, is proven in Section \ref{sec:comparison} for flat interfaces. As already announced, we require of inf/sup regularizations and some boundary regularity estimates proven in Section \ref{sec:bdry_reg}.

\subsection{Acknowledgments}

I would like to thank Ryan Hynd, Arturo Arellano, and Edgard Pimentel for their helpful feedback on this project. The author was supported by CONACyT-MEXICO grant A1-S-48577.

\section{Preliminaries}
\label{sec:preliminaries}

\subsection{Notation}

For a real number or a real-value function $\a$, we denote its positive and negative part as $\a_\pm := \max(\pm \a,0)$.

Given a point $x\in\R^n$ we use $x_1,x_2,\ldots,x_n \in \R$ to denote its coordinates. The notation $x'\in \R^{n-1}$ may be used to denote the first $(n-1)$ coordinates of $x$, i.e. $x' = (x_1,x_2,\ldots,x_{n-1})$ and $x = (x',x_n)$. Occasionally, we may also use $x' \in \R^n$ as a point which is actually in $\{x_n=0\}$. Sometimes $x_0$ may denote a fixed point in $\R^n$, or perhaps $\{x_k\}$ may be a sequence of points in $\R^n$. In such cases, the coordinates will be denoted by $(x_k)_1,(x_k)_2,\ldots,(x_k)_n$ and $x_k' =((x_k)_1,(x_k)_2,\ldots,(x_k)_{n-1})$.

The open ball in $\R^n$ of radius $r>0$ and center $x_0\in \R^n$ is denoted by $B_r(x_0)$ and we may omit the center when $x_0=0$. Whenever we talk about a ball in dimension $(n-1)$, we denote it by $B^{n-1}_r(x_0')$ or $B^{n-1}_r$.

We use $\R^{n\times n}_{\rm sym}$ to denote the space of $n\times n$ symmetric matrices. For $M_1,M_2\in \R^{n\times n}_{\rm sym}$ we say that $M_1\leq M_2$ iff $\xi\cdot M_1\xi \leq \xi\cdot M_2\xi$ for any $\xi\in \R^n\sm\{0\}$. The remaining inequalities ($\geq$, $<$, and $>$) are understood in a similar way.

We use the notation $L^\8$, $C$, $C^k$, and $C^{k,\a}$ to respectively denote the spaces of bounded functions, continuous functions, $k^{th}$-order continuously differentiable functions, and $k^{th}$-order continuously differentiable functions with derivatives of order $k$ being $\a$-H\"older continuous ($\alpha\in(0,1]$).

Given a set $\W\ss\R^n$, a relatively open subset $\G\ss\p\W$ is said to be uniformly Lipschitz regular iff for every $x_0\in \Gamma$ there exists some change of variables $\Phi:B_1\to U$ where $U=\Phi(B_1)$ is a neighborhood of $x_0=\Phi(0)$, and $\Phi$ is a bi-Lipchitz map, with a Lipschitz norm independent of $x_0$, such that
\begin{align*}
&\Phi(B_1\cap \{x_n<0\}) = U\cap (\R^n\sm \overline \W),\\
&\Phi(B_1\cap \{x_n>0\}) = U\cap \W,\\
&\Phi(B_1\cap \{x_n=0\}) = U\cap \Gamma.
\end{align*}
We will also consider the scenario where the maps are assumed to be $C^2$-regular diffeomorphism, in this case we say that $\G$ is $C^2$-regular.

A \textit{second-order operator} over $\W\ss\R^n$ is defined in terms of a function $H = H(M,p,z,x) \in C(\R^{n\times n}_{\rm sym}\times\R^n\times \R\times\W)$. For $u$ second-order differentiable at $x_0\in \W$ we compute
\[
Hu(x_0) := H(D^2u(x_0),Du(x_0),u(x_0),x_0).
\]
We may occasionally also refer to the function $H$ as the operator.

In the present article, we will say that an operator $H$ is \textit{degenerate elliptic} if for any pair $(M_1,p,z_1,x), (M_2,p,z_2,x) \in \R^{n\times n}_{\rm sym}\times\R^n\times\R\times\W$
\[
M_1\geq M_2, \quad z_1 \leq z_2 \qquad\Rightarrow\qquad H(M_1,p,z_1,x)\leq H(M_2,p,z_2,x).
\]
Our convention makes $H(M,z) = -\tr(AM)+\l z$ degenerate elliptic for any $A \in \R^{n\times n}_{sym}$ with $A\geq 0$, and $\l\geq 0$. The monotonicity with respect to $z$ is usually called properness in the literature, meanwhile the monotonicity with respect to the Hessian is the one referred to as degenerate ellipticity.

If $H=H(M,p,z)$ is independent of $x\in \W$ we say that the operator is \textit{translation invariant}.

We say the operator is \textit{convex} if for every $x\in\W$ the function $H(\cdot,\cdot,\cdot,x)$ is convex. It is instead \textit{quasi-convex} if for every $x\in\W$ and $\l\in \R$, the sub-level sets $\{H(\cdot,\cdot ,\cdot ,x)\leq \l\}$ are convex.

If $H=H(p,z,x)$ is independent of the matrix variable we say that it is a \textit{first-order operator} and
\[
Hu(x_0) := H(Du(x_0),u(x_0),x_0)
\]
can be evaluated for $u$ first-order differentiable at $x_0$.

If a first-order operator satisfies that for every $z,\l\in \R$, the set $\bigcup_{x\in \W} \{H(\cdot ,z,x)\leq \l\}$ is bounded, then we say that it has \textit{bounded sub-level sets}. Clearly the bound on the sets depends on $z$; however, as we will be considering bounded solutions we can also make the sub-level sets uniformly bounded in $z$. To be precise, let us consider that $|z|\leq M$, then $\bigcup_{(z,x)\in [-M,M]\times \W} \{H(\cdot ,z,x)\leq \l\}$ is bounded by the continuity of $H$. This property is usually a consequence of some sort of coercivity or super-linearity assumption on $H$.

A second-order operator is \textit{uniformly elliptic} if it is controlled by a family of elliptic linear operators with uniformly bounded coefficients (from above and away from zero). In order to do this it is convenient to introduce as well the \textit{extremal Pucci operators}. These are defined with respect of some interval $[\l,\L]\ss(0,\8)$ such that $\cM^\pm_{\l,\L}:\R^{n\times n}_{\rm sym}\to \R$ is given by
\begin{align*}
\cM^+_{\l,\L}(M) &:= \sup\{-\tr(AM) \ | \ A \in \R^{n\times n}_{\rm sym}, \, \l I\leq A\leq \L I\} = -\sum_{e\in \operatorname{eig}(M)} (\l e_+-\L e_-),\\
\cM^-_{\l,\L}(M) &:= \inf\{-\tr(AM) \ | \ A \in \R^{n\times n}_{\rm sym}, \, \l I\leq A\leq \L I\} = -\sum_{e\in \operatorname{eig}(M)} (\L e_+-\l e_-).
\end{align*}

We say that $H$ gives a \textit{uniformly elliptic operator} with respect to the interval $[\l,\L]\ss(0,\8)$ iff for every $(M_1,p_1,z_1,x),(M_2,p_2,z_2,x)\in \R^{n\times n}_{\rm sym}\times \R^n\times \R\times \W$, $M=M_2-M_1$, $p=p_2-p_1$, and $z=z_2-z_1$, we have that
\begin{align*}
\cM^-_{\l,\L}(M)-\L|p|-\L z_- \leq H(M_2,p_2,z_2,x) - H(M_1,p_1,z_1,x) &\leq \cM^+_{\l,\L}(M)+\L|p|+\L z_+.
\end{align*}

\subsection{Viscosity solutions}\label{sec:visc_sol}

Consider for $j\in\{1,\ldots,k\}$, degenerate elliptic operators $H_j$ defined over some subset $\W_j\ss\W$. In the following definition we give a notion of sub and super-solutions of the problem
\begin{align}\label{eq:main2}
\begin{cases}
H_1u = 0 \text{ in } \W_1,\\
\vdots\\
H_ku = 0 \text{ in } \W_k.
\end{cases}
\end{align}

We say that a function $\varphi:\W_\varphi\to \R$ \textit{touches} another function $u:\W_u\to \R$ from above at some $x_0\in \W_\varphi\cap \W_u$ iff
\[
u\leq \varphi \text{ in $\W_\varphi\cap \W_u$, with equality at $x_0$.}
\]
If additionally,
\[
u< \varphi \text{ in $\W_\varphi\cap \W_u\sm\{x_0\}$,}
\]
we say that the contact is \textit{strict}. Contact from below is defined similarly.

\begin{definition}
A function $u \in C(\W)$ is a viscosity sub-solution of \eqref{eq:main2} iff for every $\varphi \in C^2(B_r(x_0))$ that touches $u$ from above at $x_0 \in \W$, it holds that
\[
H_j\varphi(x_0) \leq 0 \text{ if } x_0 \in \W_j.
\]

A function $v \in C(\W)$ is a viscosity super-solution of \eqref{eq:main2} iff for every $\varphi \in C^2(B_r(x_0))$ that touches $v$ from below at $x_0 \in \W$, it holds that
\[
H_j\varphi(x_0) \geq 0 \text{ if } x_0 \in \W_j.
\]

Finally, $u \in C(\W)$ is a solution of \eqref{eq:main2} iff it is both a sub and super-solution of the respective problem.
\end{definition}

For sub-solutions, we could also say that the function satisfies the following inequalities in the viscosity sense
\[
\begin{cases}
H_1u \leq 0 \text{ in } \W_1,\\
\vdots\\
H_ku \leq 0 \text{ in } \W_k.
\end{cases}
\]
A similar notation is also used for super-solutions. Finally, if we ask for strict contact in the definitions above we obtain the exact same notion of solutions, a useful trick in a few proofs.

In the next definition we consider $\W_-, \Gamma, \W_+\ss\W\ss\R^n$, such that $\Gamma=\W\sm(\W_-\cup\W_+)$. The two main operators $H_-$ and $H_+$ are defined over $\W_-\cup\Gamma$ and $\W_+\cup\Gamma$ respectively. Our goal is to define the viscosity solutions for the following problem
\begin{align}\label{eq:main3}
\begin{cases}
H_-u = 0 \text{ in } \W_-,\\
H_+u = 0 \text{ in } \W_+,\\
\text{Either $H_-u$ or $H_+u = 0$ on $\Gamma$}.
\end{cases}
\end{align}
In this case the operators used for sub and super solutions are different over the set $\Gamma$.

\begin{definition}
A function $u \in C(\W)$ is a viscosity sub-solution of \eqref{eq:main3} iff it is a viscosity sub-solution of
\[
\begin{cases}
H_-u \leq 0 \text{ in } \W_-,\\
H_+u \leq 0 \text{ in } \W_+,\\
\min(H_-u,H_+u) \leq 0 \text{ on } \Gamma.
\end{cases}
\]

A function $u \in C(\W)$ is a viscosity super-solution of \eqref{eq:main3} iff it is a viscosity super-solution of
\[
\begin{cases}
H_-u \geq 0 \text{ in } \W_-,\\
H_+u \geq 0 \text{ in } \W_+,\\
\max(H_-u,H_+u) \geq 0 \text{ on } \Gamma.
\end{cases}
\]

Finally, $u \in C(\W)$ is a solution of \eqref{eq:main3} iff it is both a sub and super-solution of the respective problem.
\end{definition}

We leave as an exercise to check that the examples discussed in Section \ref{sec:exis_one_d} are viscosity solutions. Let us give a further example of a viscosity solution in an annular domain $\W=B_R\sm B_r\ss\R^n$ with $n\geq 2$ and $0<r<R$, taken from \cite{MR4249793}. The operators under consideration will be as in the introduction $H_- = |p|-1$ and $H_+=\tfrac{1}{2}(-\D)-1$.

Let $\r\in (r,R)$, $\W_-=B_R\sm B_\r$, $\W_+=B_\r\sm B_r$ and
\[
u(x) := \begin{cases}
    R-|x| \text{ in } \overline{\W_-},\\
    A+B\Phi(|x|)-|x|^2/n \text{ in } \overline{\W}\sm \overline{\W_-}.
\end{cases}
\]
The function $\Phi$ is a multiple of the fundamental solution for the Laplacian
\[
\Phi(s)=\begin{cases}
    -\ln s \text{ if } n=2,\\
    s^{2-n} \text{ if } n\geq 2.
\end{cases}
\]
The constants $\r$, $A$, and $B$ are chosen such that $u$ is continuous in $\W$, and attains the boundary value $u=0$ on $\p\W$, and $|Du|\leq 1$ on both sides of $\G$. The specific equations one should consider are the following and have infinite solutions
\begin{align*}
    &A+B\Phi(r)-r^2/n=0,\\
    &A+B\Phi(\r)-\r^2/n=R-\r,\\
    &B\Phi'(\r)-2\r/n\leq 1.
\end{align*}
See Figure \ref{fig:2}.

Under these requirements it is rather easy to check that $u$ is a viscosity solution of \eqref{eq:main} which is Lipschitz continuous but not $C^1$. One may wonder if solutions will be semi-concave, as for the eikonal equation. However, the one-dimensional examples in Section \ref{sec:exis_one_d} show that this is not necessarily the case.

\begin{figure}
    \centering
    \includegraphics[width=16cm]{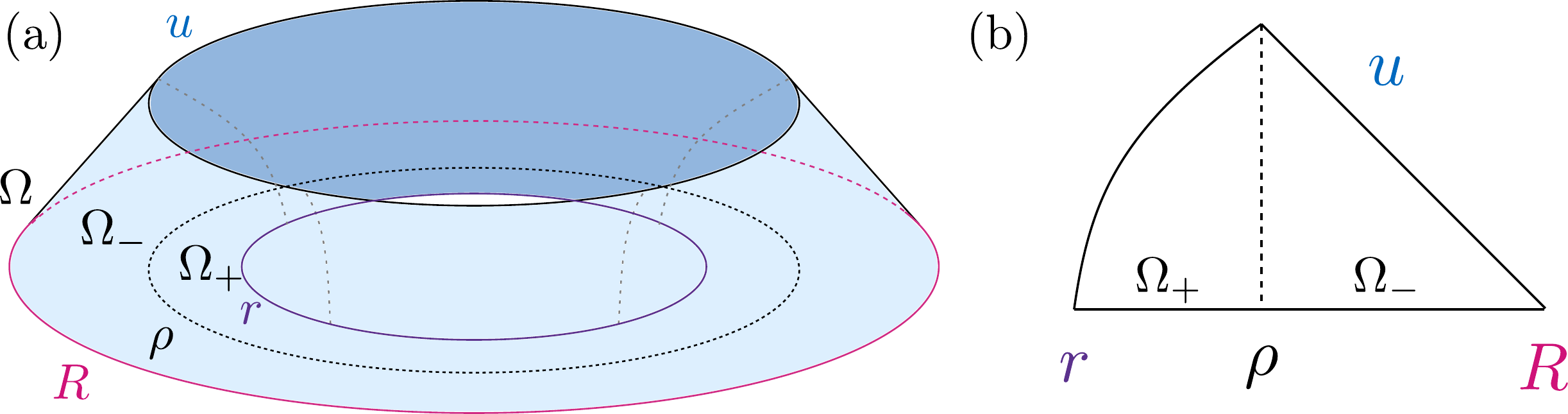}
    \caption{(a) Graph of a viscosity solution of \eqref{eq:main} in an annular domain. (b) Lateral view.}
    \label{fig:2}
\end{figure}

\subsubsection{Stability and Perron's method}
The flexibility of the problem \eqref{eq:main3} over $\Gamma$ has the benefit of providing the following stability property.

\begin{lemma}\label{lem:sta}
Let $\W_-,\W_+\ss \W\ss\R^n$ be all open sets, and let $\Gamma := \W \sm (\W_-\cup\W_+)$. Let $H_\pm$ be degenerate elliptic operators over $\W_\pm\cup\Gamma$ respectively, and let $\{u_k\} \ss C(\W)$ be a sequence of viscosity sub-solutions for \eqref{eq:main3} that converges uniformly to some $u\in C(\W)$. Then $u$ is also a viscosity sub-solution for \eqref{eq:main3}.
\end{lemma}

\begin{proof}
Let $\varphi \in C^2(B_r(x_0))$ be a test function that strictly touches $u$ from above at $x_0\in\W$. Up to a vertical translation and a sub-sequence, $\varphi+m_{k}$ touches $u_{k}$ from above at $x_{k}\to x_0$, and $m_k\to0$, due to the uniform convergence.

If $x_0\in \W_+$ then we can assume without loss of generality that $\{x_{k}\} \ss \W_+$ as well. This means that $H_+(\varphi+m_k)(x_k) \leq 0$ and the desired inequality for $H_+\varphi(x_0)$ follows by continuity. The same reasoning applies if $x_0\in \W_-$ instead.

If $x_0\in \Gamma$, then either $H_-(\varphi+m_k)(x_k)\leq 0$ or $H_+(\varphi+m_k)(x_k)\leq 0$ must be true an infinite number of times. In either case we conclude once again by continuity that $\min(H_-\varphi(x_0),H_+\varphi(x_0))\leq 0$.
\end{proof}

The following consequence of the stability given by the previous lemma is the first step towards the construction of the viscosity solution of the boundary value problem by Perron's method.

\begin{corollary}\label{cor:perron}
Let $\W_\pm\ss \W\ss\R^n$ be all open sets, and let $\Gamma := \W \sm (\W_-\cup\W_+)$. Let $H_\pm$ be degenerate elliptic operators over $\W_\pm\cup\Gamma$ respectively. Let $\mathcal S$ be a set of viscosity sub-solutions for the problem \eqref{eq:main3} which is equicontinuous and bounded. Then the upper envelope $u_\mathcal S(x) := \sup_{u\in \mathcal S}u(x)$ is also a sub-solution of \eqref{eq:main3}.
\end{corollary}

In the case where $\mathcal S$ is finite, the equicontinuity and boundedness hypotheses are superfluous and the result is known as the lattice property: The maximum of two sub-solutions is a sub-solution.

The following existence result does not require the comparison principle to hold. However, it assumes as in the previous lemma that the upper envelope of the family of sub-solutions of the Dirichlet problem turns out to be continuous up to the boundary.

\begin{lemma}[Perron's Solution]
\label{cor:perron2}
Let $\W_\pm\ss \W\ss\R^n$ be all open sets such that $\p\W\neq \emptyset$, and let $\Gamma := \W \sm (\W_-\cup\W_+)$. Let $H_\pm$ be degenerate elliptic operators over $\W_\pm\cup\Gamma$ respectively. Given $g \in C(\partial\W)$, define the set $\mathcal S_g$ such that
\[
\mathcal S_g := \{u \in C(\overline\W)\ | \ \text{$u$ is a viscosity sub-solution of \eqref{eq:main3} and $u = g$ in $\p\W$}\}.
\]
If there exists an equicontinuous and bounded subset $\mathcal S\ss\mathcal S_g$ such that
\[
u_{g}(x) := \sup_{u\in\mathcal{S}_g}u(x) = \sup_{u\in\mathcal{S}}u(x),
\]
then $u_g$ is a viscosity solution of \eqref{eq:main3} on $\W$ taking the boundary value $g$ continuously on $\partial\W$.
\end{lemma}

Notice that if $\W$ is bounded, as we will set in the next section, the boundedness condition in $\mathcal S$ follows from the equicontinuity assumption.

\begin{proof}
By Corollary \ref{cor:perron} we know that $u_g$ is indeed a sub-solution. In this proof we will just check that $u_g$ is also a super-solution.

Assume by contradiction that $\varphi \in C^2(\overline{B_r(x_0)})$ strictly touches $u_g$ from below at some $x_0\in \Gamma$ over $B_r(x_0)\ss\W$, but nonetheless $\max(H_-\varphi,H_+\varphi)<0$ over $B_r(x_0)$. Although we focus on the case where the contact occurs over the interface, the following argument can also be adapted to the case where $x_0$ belongs to either $\W_-$ or $\W_+$.

Letting $\d\in (0,\max_{\p B_r(x_0)}(u_g-\varphi))$, so that we still have $\max(H_-(\varphi+\d),H_+(\varphi+\d))<0$ over $B_r(x_0)$, we construct
\[
v := \begin{cases}
    \max(u_g,\varphi+\d) \text{ in } B_r(x_0),\\
    u_g \text{ in } \overline \W \sm B_r(x_0).
\end{cases}
\]
This is a function in $\mathcal S_g$ that contradicts the maximality of $u_g$.
\end{proof}

\begin{remark}
One can easily extend the notions of viscosity solutions to semi-continuous functions. In this case there are a similar stability properties under half-relaxed limits or gamma-convergence. These will not be used in our work as the regularity of the operators will allow us to assume that the solutions are always continuous.
\end{remark}

\subsubsection{Compactness for sub-solutions}\label{comp_sub_sol}

Just as sub-solutions of the eikonal equation are Lipschitz continuous, we can show that viscosity sub-solutions of first-order, degenerate elliptic operators have a modulus of continuity that depends on the zero sub-level set of the operator, which we assume to be bounded.

\begin{lemma}\label{lem:lips1}
    Let $\W\ss\R^n$ be an open bounded set with uniformly Lipschitz regular boundary. Let $H\in C(\R^n\times\R\times \W)$ be first-order, degenerate elliptic operator with bounded sub-level sets. Then any bounded sub-solution $u \in C(\W)$ of $Hu\leq 0$ in $B_r(x_0)$ is Lipschitz continuous.
\end{lemma}

\begin{proof}
    Follows because $u$ is also a sub-solution of the eikonal equation $|Du|-R \leq 0$ in $\W$ for some radius $R>0$ sufficiently large such that $\bigcup_{(z,x)\in [-M,M]\times\W}\{H(\cdot,z,x)\leq0\}\ss B_R$ with $M:=\|u\|_{L^\8(\W)}$. This observation allows to show first that $[u]_{C^{0,1}(\W)}\leq R$ if $\W$ is convex. In general, the Lipschitz assumption on $\p\W$ suffices to extend the Lipschitz regularity to the whole domain by a covering argument. A detailed discussion can be found in \cite[Section IV.3.1]{MR1484411}.
\end{proof}

In the following lemma, we consider sub-solutions of \eqref{eq:main3} with continuous boundary data on $\W$. We then perform a replacement over $\W_+$ using the solution of the problem $H_+u = 0$ in $\W_+$. A sufficient condition for this construction to be well-defined is that $H_+$ is uniformly elliptic and $\W_+$ is a bounded set with uniformly Lipschitz regular boundary (also an exterior cone condition would be enough). This ensures the existence of barriers controlling the behavior at the boundary, as shown in \cite{MR221087}. The result is a continuous function on $\overline \W$ with a modulus of continuity depending on the problem's data.

\begin{lemma}\label{lem:lift}
    Let $\W_+\ss \W\ss\R^n$ be open bounded sets, $\W_-= \W\sm\overline{\W_+}$, and $\Gamma=\W\sm(\W_+\cup\W_-)=\p\W_\pm\cap\W$. Assume that $\W_\pm$ have uniformly Lipschitz regular boundaries. Consider $H_\pm$ degenerate elliptic operators defined over $\W_\pm\cup\Gamma$ respectively, such that $H_-$ is a first-order operator with bounded sub-level sets, while $H_+$ is a uniformly elliptic second-order operator. Let $g \in C(\p\W)$ be the boundary data of the problem.
    
    Given $u \in C(\overline \W)$, a bounded sub-solution of \eqref{eq:main3} such that $u=g$ over $\p\W$, consider $\bar u \in C(\overline \W)$ to be the solution of the Dirichlet boundary value problem
    \[
    \begin{cases}
        H_+\bar u = 0 \text{ in } \W_+,\\
        \bar u = u \text{ on } \W\sm \W_+.
    \end{cases}
    \]
    Then $\bar u$ is well defined and its modulus of continuity depends only on $\W$, $\W_\pm$, $H_\pm$, and $g$. Moreover, $\bar u$ is a sub-solution of \eqref{eq:main3}.
\end{lemma}

We sketch the main arguments in the proof which combines some of the fundamental constructions and estimates for uniformly elliptic equations developed by several authors.

\begin{proof}
By Lemma \ref{lem:lips1} we get that $u$ is Lipschitz on $\W\sm \W_+$. From this and the existence of barriers provided in \cite{MR221087} we obtain the existence of viscosity solution by Perron's method \cite[Section 4]{MR1118699} and \cite{MR894587}. In other words, $\bar u$ is well defined.

The uniform modulus of continuity for the solution is a consequence of the interior Hölder estimates for uniformly elliptic equations \cite{MR1005611,MR1048584}, and the existence of barriers with uniform modulus, once again found in the explicit construction given in \cite{MR221087}. The procedure for combining the modulus of continuity at the boundary with the interior regularity estimate is standard. See for example, the end of Chapter 4 in \cite{MR1351007}.

To see that $\bar u$ is a sub-solution, we just have to consider the situation in which a test function touches $\bar u$ from above over the interface, being the other two cases immediate by construction. In this scenario, the test function also touches the original sub-solution $u$, and the desired inequality follows promptly.
\end{proof}

By combining the previous result with Perron's method stated in Lemma \ref{cor:perron2} we recover the existence of solutions to the Dirichlet problem.

\begin{theorem}\label{thm:existence}
    Let $\W_+\ss \W\ss\R^n$ be open bounded sets, $\W_-= \W\sm\overline{\W_+}$, and $\Gamma=\W\sm(\W_+\cup\W_-)=\p\W_\pm\cap\W$. Assume that $\W_\pm$ have uniformly Lipschitz regular boundaries. Consider $H_\pm$ as degenerate elliptic operators defined over $\W_\pm\cup\Gamma$ respectively, such that $H_-$ is a first-order operator with bounded sub-level sets, while $H_+$ is a uniformly elliptic second-order operator.
    
    Given $g \in C(\overline \W)$ being a sub-solution of \eqref{eq:main3}, there exists a viscosity solution of \eqref{eq:main3} in $\W$ taking the boundary value $g$ on $\p\W$.
\end{theorem}

For instance, the function $g=0$ serves as a sub-solution for the original problem discussed in the introduction.

\begin{corollary}\label{cor:existence}
    Let $\W_+\ss \W\ss\R^n$ be open bounded sets, $\W_-= \W\sm\overline{\W_+}$, and $\Gamma=\W\sm(\W_+\cup\W_-)=\p\W_\pm\cap\W$. Assume that $\W_\pm$ have uniformly Lipschitz regular boundaries. Then there exists a viscosity solution $u\in C(\overline\W)$ of \eqref{eq:main} with $u=0$ on $\p\W$.
\end{corollary}

In the Section \ref{sec:exis_one_d} we discussed the one dimensional problem over $\W=(-1,1)$, with $\W_+ = (0,1)$. 
Recall that in this case, when we set $g(-1)=0$ and $g(1)<-2$ there are no solutions to the boundary value problem. In a few words, any pair of sub-solutions in $\W_\pm$, with the given boundary values, and coinciding at the interface $\Gamma = \{0\}$, will not satisfy the sub-solution property at zero.

From this point the main goal of the article is to establish uniqueness of solutions for the Dirichlet boundary value problem.

\section{Strong equation}
\label{sec:strong}

In this section we find a stronger characterization for the problem \eqref{eq:main3} with a $C^2$-regular interface $\Gamma$, and $H_+$ uniformly elliptic. Namely, the solutions of \eqref{eq:main3} are exactly the solutions of
\begin{align}\label{eq:main4}
\begin{cases}
H_- u = 0 \text{ in } \W_-\cup \Gamma,\\
H_+ u = 0 \text{ in } \W_+.
\end{cases}
\end{align}

\begin{theorem}\label{thm:strong}
Let $\W_+\ss \W\ss\R^n$ be open sets, $\W_-= \W\sm\overline{\W_+}$, and $\Gamma=\W\sm(\W_+\cup\W_-)=\p\W_\pm\cap\W$ a $C^2$-regular interface. Let $H_\pm$ be degenerate elliptic operators over $\W_\pm\cup\Gamma$, with $H_+$ uniformly elliptic. Then $u\in C(\W)$ is a viscosity sub-solution (super-solution or solution) of the relaxed problem \eqref{eq:main3} if and only if it is a viscosity sub-solution (super-solution or solution) of the strong problem \eqref{eq:main4}.
\end{theorem}

After a local $C^2$-regular change of variables we can assume without loss of generality that $\W=B_1$, $\W_\pm = B_1\cap\{\pm x_n>0\}$ and $\Gamma = B_1\cap\{x_n=0\}$. This is an assumption adopted in the rest of the section.

As it will be shown in our proof, we only need a coercivity hypothesis on $H_+$ with respect to the Hessian variable. Namely that for any $(M,p,z,x)\in \R^{n\times n}_{\rm sym}\times\R^n\times\R\times\W$ there exists some $\a> 0$ such that\footnote{For $v,w\in\R^n$, we denote $v\otimes w \in \R^{n\times n}$ such that $(v\otimes w)_{ij} = v_iw_j$.}
\[
H_+(M-\a (e_n\otimes e_n),p,z,x) >0.
\]
This follows from a uniform ellipticity hypothesis because of the inequality
\[
H_+(M-\a (e_n\otimes e_n),p,z,x) - H_+(M,p,z,x) \geq \cM^-_{\l,\L}(-\a(e_n\otimes e_n)) = \a\l.
\] 

In the following lemmas we denote by $\varphi_\pm$ the restrictions of a given function $\varphi:\W\ss\R^n\to \R$ to $\W_\pm\cup \Gamma$ respectively. The hypothesis $\varphi \in C(\W)\cap C^2(\W_-\cup\Gamma)\cap C^2(\W_+\cup\Gamma)$ means that $\varphi$ is continuous across $\G$, and each one of the restrictions $\varphi_\pm$ are $C^2$-regular up to $\Gamma$.

\begin{lemma}\label{lemma1}
Let $\W=B_1$, $\W_\pm=B_1\cap\{\pm x_n>0\}$, and $\Gamma=B_1\cap\{x_n=0\}$. Let $H_\pm$ be degenerate elliptic operators over $\W_\pm\cup\Gamma$, with $H_+$ uniformly elliptic. Then, for any $\varphi \in C(\W)\cap C^2(\W_-\cup\Gamma)\cap C^2(\W_+\cup\Gamma)$ such that
\[
\min(H_-\varphi_+(0),\p_n \varphi_-(0)-\p_n\varphi_+(0))>0,
\]
there exists a test function $\psi \in C^2(B_\r)$, with $\r\in(0,1)$, that touches $\varphi$ from above at the origin and satisfies
\[
\min(H_-\psi(0),H_+\psi(0))>0.
\]
\end{lemma}

We would like to emphasize the geometric interpretation of the inequality $\p_n \varphi_-(0)-\p_n\varphi_+(0)>0$. Specifically, this condition implies that the graph of $\varphi$ forms a concave angle around the origin and along $\G$.

\begin{proof}
Given $\e\in(0,1)$, and $\a>0$, consider the functions
\begin{align*}
P_+(x) &:= \varphi(0) + D\varphi_+(0)\cdot x + \tfrac{1}{2} x\cdot D^2\varphi_+(0)x + \tfrac{\e}{2}|x|^2,\\
P_-(x) &:= \varphi(0) + D\varphi_-(0)\cdot x + \frac{1}{2} x\cdot D^2\varphi_-(0)x + \tfrac{\e}{2}|x|^2,\\
\psi(x) &:= P_+(x) + \e x_n - \a x_n^2.
\end{align*}
By having that $\varphi_+$ and $\varphi_-$ are second-order differentiables at the origin we have that there exists some $\r_0 \in(0,1)$ sufficiently small such that
\[
\begin{cases}
\varphi_+ \leq P_+\text{ in } B_{\r_0}\cap \{x_n\geq 0\},\\
\varphi_- \leq P_-\text{ in }B_{\r_0}\cap \{x_n\leq 0\}.
\end{cases}
\]

For $\e\in(0,1)$ sufficiently small we get by continuity and the degenerate ellipticity that
\[
H_-\psi(0)>0 \qquad\text{ and } \qquad \p_n \varphi_-(0)-\p_n\varphi_+(0)-2\e>0.
\]
Take now $\a$ sufficiently large such that, thanks to the uniform ellipticity, we can enforce $H_+\psi(0)> 0$. Hence the proposed $\psi$ already satisfies the desired inequality,
\[
\min(H_-\psi(0), H_+\psi(0))>0.
\]

To finish the proof we need to show that $\psi$ touches $\varphi$ from above at the origin over some sufficiently small ball $B_\r$. This follows if for some $\r \in(0,\r_0)$ it holds that
\[
\begin{cases}
P_+\leq \psi \text{ in } B_\r \cap \{x_n>0\},\\
P_-\leq \psi \text{ in } B_\r \cap \{x_n<0\}.
\end{cases}
\]

The inequality over $B_\r \cap\{x_n>0\}$ is equivalent to checking that $\e x_n \geq \a x_n^2$. This is clearly the case once we take $\r\leq \e/\a$.

To prove the inequality over $B_\r \cap \{x_n<0\}$, we need to show that $\e x_n \geq P_--P_++\a x_n^2$ holds over the same domain. Here, we should keep in mind that $P_--P_+ = x_n R$, where $R(x)$ is defined as follows
\[
R(x) := (\p_n\varphi_-(0)-\p_n\varphi_+(0))+ \sum_{j=1}^{n-1} (\p_{nj}\varphi_-(0)-\p_{nj}\varphi_+(0))x_j + \frac{1}{2}(\p_{nn}\varphi_-(0)-\p_{nn}\varphi_+(0)) x_n.
\]
We use that $\p_n\varphi_-(0)-\p_n\varphi_+(0)>2\e$ and bound the quadratic terms by $C\r|x_n|$, then $P_--P_++\a x_n^2 \leq 2\e x_n + C\r|x_n|$. This quantity gets bounded by $\e x_n$ if we choose $\r\leq \e/C$.
\end{proof}

As a consequence of the previous lemma we see that the transmission condition only takes into consideration the uniform ellipticity of $H_+$ and ignores the rest of its structure. The following corollary gives an intermediate formulation, still in the viscosity sense spirit, between the problems \eqref{eq:main3} and \eqref{eq:main4}. We skip the proof as it follows directly from the definitions.

\begin{corollary}\label{cor:strong_eq}
Let $\W=B_1$, $\W_\pm=B_1\cap\{\pm x_n>0\}$, and $\Gamma=B_1\cap\{x_n=0\}$. Let $H_\pm$ be degenerate elliptic operators over $\W_\pm\cup\Gamma$, with $H_+$ uniformly elliptic. Let $u\in C(\W)$ be a viscosity sub-solution of \eqref{eq:main3}. Then for every
\[
\varphi \in C(B_r)\cap C^2(B_r\cap(\W_-\cup\Gamma))\cap C^2(B_r\cap(\W_+\cup\Gamma))
\]
that touches $u$ from above at $0\, (\in \Gamma)$ it holds that
\[
\min(H_-\varphi_+(0),\p_n\varphi_-(0)-\p_n\varphi_+(0)) \leq 0.
\]
\end{corollary}

This following lemma makes the final connection with the strong problem \eqref{eq:main4}.

\begin{lemma}\label{lemma1.5}
Let $\W=B_1$, $\W_\pm=B_1\cap\{\pm x_n>0\}$, and $\Gamma=B_1\cap\{x_n=0\}$. Let $H_\pm$ be degenerate elliptic operators over $\W_\pm\cup\Gamma$, with $H_+$ uniformly elliptic. Let $u\in C(\W)$ be such that for every
\[
\varphi \in C(B_r(x_0))\cap C^2(B_r(x_0)\cap(\W_-\cup\Gamma))\cap C^2(B_r(x_0)\cap(\W_+\cup\Gamma))
\]
that touches $u$ from above at $x_0\in \W$ it holds that
\[
\begin{cases}
H_-\varphi(x_0) \leq 0 \text{ if } x_0 \in \W_-,\\
H_+\varphi(x_0)  \leq 0 \text{ if } x_0 \in \W_+,\\
\min(H_-\varphi_+(x_0),\p_n\varphi_-(x_0)-\p_n\varphi_+(x_0)) \leq 0 \text{ if } x_0 \in \Gamma.
\end{cases}
\]
Then $u$ is a viscosity sub-solution of \eqref{eq:main4}.
\end{lemma}

\begin{proof}
Let $\varphi \in C^2(B_r(x_0))$ touching $u$ from above at $x_0 \in \W$. The verification of the sub-solution condition reduces to check the case $x_0\in \Gamma$. Assume by contradiction that $H_-\varphi>0$ over $B_{\r_1}(x_0)\cap (\W_-\cup\G)$ for some small $\r_1 \in (0,r)$.

Let
\[
\psi(x) := \varphi(x) + \e\eta(|x'-x_0|/\r_1)x_n(x_n+\r_1/2)
\]
where $\eta \in C^\8_0((-1/2,1/2))$ is even and non-negative, with $\eta=1$ in $(-1/4,1/4)$. The value of $\e>0$ is chosen sufficiently small such that $H_-\psi>0$ over $B_{\r_1}(x_0)\cap (\W_-\cup\G)$.

By construction $\psi \geq \varphi$ in $\p (B_{\r_1}(x_0) \cap \W_-)=(\p B_{\r_1}(x_0) \cap \{x_n<0\})\cup (B_{\r_1}(x_0) \cap \{x_n=0\})$, and $\psi(x_0) = \varphi(x_0)$. So, either $\psi$ continues touching $u$ form above at $x_0$, or there is some $c>0$ such that $\psi+c$ touches $u$ from above at some point $x^* \in B_{\r_1}(x_0) \cap \W_-$. However, this last alternative would contradict the main hypothesis on $u$, because $\psi+c$ is a test function for which $H_-(\psi+c)(x^*) \geq H_-\psi(x^*) > 0$ given that $x^* \in B_{\r_1}(x_0)\cap \W_-$.

Hence we must have that $\psi$ continues touching $u$ form above at $x_0$. We construct in this way the following test function over $B_{\r_1/4}(x_0)$
\[
\phi(x) := \begin{cases}
\psi(x) \text{ if } x \in B_{\r_1/4}(x_0)\cap\{x_n\leq 0\},\\
\varphi(x) \text{ if } x \in B_{\r_1/4}(x_0)\cap\{x_n > 0\}.
\end{cases} 
\]
Given that
\[
\p_n \phi_-(x_0) = \p_n \varphi(x_0) + \e\r_1/2 > \p_n \varphi(x_0) = \p_n \phi_+(x_0),
\]
we must have from the main hypothesis on $u$ that $H_-\phi_+(x_0) \leq 0$. However this goes against our initial assumption given that $H_-\phi_+(x_0)=H_-\varphi(x_0) >0$.
\end{proof}

The previous lemmas also have analogues for super-solutions with similar proofs. As a consequence we get the proof of Theorem \ref{thm:strong} in the case of a flat interface. The general scenario with $C^2$ interfaces follows by a standard change of variables argument around each point of the interface where the equation is tested. The $C^2$ regularity of the change of variables suffices to get operators with the same hypothesis as in the flat case.

\section{Comparison Principle}
\label{sec:comparison}

In this section we consider the open and bounded sets $\W\ss\R^n$, $\W_\pm = \W\cap\{\pm x_n >0 \}$, and the \underline{flat} interface $\Gamma=\W\cap\{x_n=0\}$. These are the hypotheses on the operators with respect to some fixed uniform ellipticity parameters $[\l,\L]\ss(0,\8)$:
\begin{itemize}
    \item[\namedlabel{hyp:hm}{($H_-$)}] Let $H_-\in C^{0,1}_{loc}(\R^n\times\R\times(\W_-\cup\Gamma))$ be a first-order quasi-convex operator with bounded sub-level sets.
    \item[\namedlabel{hyp:hp}{($H_+$)}] Let $H_+\in C(\R^{n\times n}_{\rm sym}\times\R^n\times\R\times(\W_+\cup\Gamma))$ be degenerate elliptic of the form
    \[
    H_+(M,p,z,x) = F(M,p,z)+f(z,x),
    \]
    where $F \in C(\R^{n\times n}_{\rm sym}\times\R^n\times \R)$ is translation invariant and uniformly elliptic with respect to the interval $[\l,\L]$; and $f \in C(\R\times(\W_+\cup\Gamma))$ is increasing in the $z$ variable.
\end{itemize}

Our hypotheses on the operators include a large class of examples, in particular we can treat the Poisson and eikonal equations which appear in the introduction. However, our method of proof can not say much about diffusions that may depend on the position (i.e. $H_+(M,x) = -\tr(A(x)M)$), not even in the case of smooth coefficients.

Notice as well that the hypotheses on each operator guarantee that each one separately enjoys a comparison principle \cite{MR1118699,MR1351007}.

Finally, we consider the comparison principle for two equations separated by some gap $\eta\geq 0$. For $u,-v \in C(\overline\W)\cap C^{0,1}(\W_-\cup\Gamma)$, we have the following assumptions:
\begin{itemize}
    \item[\namedlabel{hyp:sub}{(Sub)}] The function $u$ is a viscosity sub-solution of
    \[
    \begin{cases}
    H_-u + \eta \leq 0 \text{ in } \W_-\cup\Gamma,\\
    H_+u + \eta \leq 0 \text{ in } \W_+.
    \end{cases}
    \]
    \item[\namedlabel{hyp:sup}{(Sup)}] The function $v$ is a viscosity super-solution of
    \[
    \begin{cases}
    H_-v \geq 0 \text{ in } \W_-\cup\Gamma,\\
    H_+ v \geq 0 \text{ in } \W_+.
    \end{cases}
    \]
\end{itemize}

\begin{theorem}\label{thm:comp}
Let $\W\ss\R^n$ open and bounded, $\W_\pm = \W\cap\{\pm x_n >0 \}$, and $\Gamma=\W\cap\{x_n=0\}$. Let $H_\pm$ be degenerate elliptic operators that satisfy \ref{hyp:hm} and \ref{hyp:hp} respectively. Let $u,v \in C(\overline\W)\cap C^{0,1}(\W_-)$ satisfy \ref{hyp:sub} and \ref{hyp:sup} respectively with $\eta>0$. Then
\[
u\leq v \text{ on } \p\W \qquad\Rightarrow\qquad u\leq v \text{ in } \W.
\]
\end{theorem}

Under some additional conditions it is possible to remove the gap between the equations. For instance:

\begin{itemize}
    \item[\namedlabel{hyp:monz}{(Mz)}] For some $\l>0$,
    \[
    h\geq 0 \qquad \Rightarrow\qquad H_\pm(p,z+h,x) \geq H_\pm(p,z,x) + \l h.
    \]
    \item[\namedlabel{hyp:monp}{(Mp)}] For some $\l>0$, and $e\in \p B_1$
    \[
    h\geq 0 \qquad \Rightarrow\qquad H_\pm(p+he,z,x)\geq H_\pm(p,z,x)+ \l h.
    \]
    \item[\namedlabel{hyp:con1}{(C)}] $H_\pm$ are convex operators independent of $z$, and there exist $\b\in C^2(\overline\W)$ and $\bar\eta>0$ such that
    \[
    \begin{cases}
        H_-\b + \bar \eta \leq 0 \text{ in } \W_-\cup \G,\\
        H_+\b + \bar\eta \leq 0 \text{ in } \W_+.
    \end{cases}
    \]
\end{itemize}

\begin{remark}\label{rmk:other}
    The first case is usually applied to operators of the form $H_-(p,z,x)=\l z + H(p,x)$ appearing in optimal control problem with geometric discount. The second case can be applied to show a comparison principle in time evolving problems, being $e$ the time direction. Finally, the last case applies to our model problem with the eikonal operator $H_-(p)=|p|-1$ and $H_+(M)=\tfrac{1}{2}(-\tr M)-1$ by taking $\b = 0$ and $\bar \eta=1$.
\end{remark}

\begin{corollary}\label{cor:comp}
Let $\W\ss\R^n$ open and bounded, $\W_\pm = \W\cap\{\pm x_n >0 \}$, and $\Gamma=\W\cap\{x_n=0\}$. Let $H_\pm$ be degenerate elliptic operators that satisfy \ref{hyp:hm} and \ref{hyp:hp} respectively and at least of the hypotheses \ref{hyp:monz}, \ref{hyp:monp}, or \ref{hyp:con1}. Let $u,v \in C(\overline\W)\cap C^{0,1}(\W_-)$ satisfy \ref{hyp:sub} and \ref{hyp:sup} respectively with $\eta=0$. Then
\[
u\leq v \text{ on } \p\W \qquad\Rightarrow\qquad u\leq v \text{ in } \W.
\]
\end{corollary}

In any of the three cases the idea is to approximate the sub-solution with a limiting sequence that satisfies the hypothesis of Theorem \ref{thm:comp} with a positive gap. Let us illustrate the last one, hypothesis \ref{hyp:con1}, pertinent to the optimal control problem discussed in the introduction.

Given that $H_\pm$ are independent of $z$ we may translate the barrier and assume that $\b\leq 0$ on $\p\W$. Assume now that $u,v \in C(\overline\W)\cap C^{0,1}(\W_-)$ satisfy \ref{hyp:hm}, \ref{hyp:hp} with $\eta=0$, and $u\leq v$ on $\p\W$. Then $u_t = (1-t)u+t\b$ and $v_t=v$ satisfy \ref{hyp:hm} and \ref{hyp:hp} with $\eta=t\bar\eta$ ($t\in(0,1)$), and still $u_t\leq v$ on $\p\W$. Hence $u_t\leq v_t$ in $\W$ and the desired conclusion for $u$ and $v$ gets recovered in the limit as $t\to 0^+$.

\begin{corollary}
    Let $\W\ss\R^n$ open and bounded, $\W_\pm = \W\cap\{\pm x_n >0 \}$, and $\Gamma=\W\cap\{x_n=0\}$. Assume that $\W_\pm$ have uniformly Lipschitz regular boundaries. Let $H_\pm$ be degenerate elliptic operators that satisfy \ref{hyp:hm} and \ref{hyp:hp} respectively and at least of the hypotheses \ref{hyp:monz}, \ref{hyp:monp}, or \ref{hyp:con1}.
    
    Given $g \in C(\overline \W)$ being a sub-solution of \eqref{eq:main3}, there exists a \underline{unique} viscosity solution of \eqref{eq:main3} in $\W$ taking the boundary value $g$ on $\p\W$.
\end{corollary}

The proof of Theorem \ref{thm:comp} goes by contradiction, here is a summary of the steps involved: Assume that $u\leq v$ over the boundary $\p\W$, nevertheless $\{u>v\}$ is non-empty. We use the hypothesis on the operators and the regularity of the solutions to replace the solutions by some inf/sup regularizations (along directions parallel to the interface) preserving all the other hypothesis and gaining semi-convexity along the interface.

By translating $v$ upwards (denoted the same) we can then assume that $v$ touches $u$ from above at some point $x_0\in \G$ (being the other cases clearly not possible by the classical comparison principles). At this point we also replace $u$ and $v$ over $\W_+$ with exact solutions of the uniformly elliptic equation. Finally, we observe that over $\G$ the solutions are trapped between two paraboloids, from this control and the equations on both sides, we can extract further regularity that finally allow us to reproduce the argument in Section \ref{sec:com_one_d} for the one dimensional case. The boundary estimates are discussed in Section \ref{sec:bdry_reg}.

\subsection{Inf/sup convolutions}\label{sec:inf_sup}

Here is a quick review of the inf/sup convolutions and its properties. The main difference is that we only perform the regularization in directions parallel to $\{x_n=0\}=\R^{n-1}$.

Given $\W\ss\R^n$ and $r>0$ we denote
\[
\W^r := \{x\in \W\ | \ \overline{B_r(x)} \ss \W\}.
\]

\begin{definition}\label{convol}
Let $\W\ss\R^n$, $u \in C(\W)\cap L^\8(\W)$, $\e\in(0,1)$, $M=\|u\|_{L^\8(\W)}$, and $r= 2\sqrt{M\e}$.

We define sup-convolution $u^\e:\W^r \to \R$ such that
\begin{align*}
u^\e(x) &:= \sup\{ u(y',x_n)-\tfrac{1}{2\e}|y'-x'|^2\ | \ (y',x_n)\in \W\}.
\end{align*}

We define the inf-convolution $u_\e:\W^r \to \R$ such that
\begin{align*}
u_\e(x) &:= \inf\{ u(y',x_n)+\tfrac{1}{2\e}|y'-x'|^2\ | \ (y',x_n)\in \W\}.
\end{align*}
\end{definition}

By choosing $x\in \W^r$, with $r= 2\sqrt{M\e}$ and $M=\|u\|_{L^\8(\W)}$, we guarantee that the supremum and infimum above are actually maximum and minimum attained at some point over the disc $\{(y',x_n) \ | \ |y'-x'| \leq r\}$.

The proof of the following lemma is a standard adaptation of the ideas in \cite[Chapter 5]{MR1351007}. As usual there is a analog version for super-solutions and their corresponding inf-convolution.

\begin{lemma}
Let $\W\ss\R^n$, $u \in C(\W)\cap L^\8(\W)$, $\e\in(0,1)$, $M=\|u\|_{L^\8(\W)}$, and $r= 2\sqrt{M\e}$. In the last two properties we will also denote $\W^r_\pm=\W^r\cap\{\pm x_n>0\}$, and $\G^r = \W^r\cap \{x_n=0\}$. The following properties hold:
\begin{itemize}
\item \textbf{Semi-convexity:} For each $h \in \R$ such that $\W^r\cap \{x_n=h\}\neq \emptyset$, the function $x'\in \{x'\in \R^{n-1}\ | \ (x',h)\in \W^r\} \mapsto u^\e(x',h) + \tfrac{1}{2\e}|x'|^2$ is convex.

\item \textbf{Approximation:} The functions $u^\e$ decrease to $u$ locally uniformly as $\e\searrow 0$.

\item \textbf{Lipschitz solutions for first-order Lipschitz equations:} Given $H_-\in C^{0,1}_{loc}(\R^n\times \R\times \W_-\cap \G)$ and $u \in C^{0,1}(\W_-\cap \G)$ there exists a constant $C$ such that if $u$ is a viscosity sub-solution of the problem
\[
H_-u \leq 0 \text{ in } \W_-\cap \G
\]
we also get that $u^\e$ is a sub-solution of
\[
H_-u - C\e \leq 0 \text{ in } \W_-^r\cap \G^r
\]

\item \textbf{Solutions for translation invariant equations:} Let $H_+\in C(\R^{n\times n}_{sym}\times\R^n\times \R\times \W_+)$ such that
\[
H_+(M,p,z,x) = F(M,p,z)+f(z,x)
\]
with $F\in C(\R^{n\times n}_{sym}\times\R^n\times \R)$ translation invariant and $f\in C(\R\times \W_+)$, both degenerate elliptic. Given $u \in C(\W_+\cap \G)$, there exists a modulus of continuity $\w$ such that if $u$ is a viscosity sub-solution of the problem
\[
H_+u \leq 0 \text{ in } \W_+,
\]
we also get that $u^\e$ is a sub-solution of
\[
H_+u - \w(\e) \leq 0 \text{ in } \W_+^r.
\]
\end{itemize} 
  
\end{lemma}

\subsection{Regularity estimates}\label{sec:bdry_reg}

The regularity results in this section use as hypotheses the semi-convexity that results from the inf/sup-convolutions.

\subsubsection{Uniformly elliptic regime}

The following lemma is closely related to the $C^{1,\a}$ boundary regularity estimate due to Krylov \cite{MR661144}. Its proof is part of an argument due to Caffarelli \cite[Lemma 4.31]{MR787227} and \cite[proof of Theorem 9.31]{Gilbarg-Trudinger2001} using the Harnack inequality.

\begin{lemma}\label{lem:44}
Let $[\l,\L]\ss(0,\8)$, $\e\in(0,1)$, $p'\in \R^{n-1}$ with $|p'|\leq \L$, and let $u\in C(\overline{B_1\cap\{x_n> 0\}})$ satisfy 
\[
\begin{cases}
\cM^+_{\l,\L}(D^2u) + \L|Du| + \L u_+ + \L \geq 0 \text{ in } B_1\cap\{x_n>0\},\\
\cM^-_{\l,\L}(D^2u) - \L|Du| - \L u_- - \L \leq 0 \text{ in } B_1\cap\{x_n>0\},\\
u\geq -\tfrac{1}{2\e}|x'|^2+ p'\cdot x' \text{ on } B_1\cap\{x_n=0\}.
\end{cases}
\]
Then one of the following alternatives must be true:
\begin{enumerate}
\item For every $p_n \in \R$ there exists a sufficiently small radius $r\in(0,1)$ such that
\[
u\geq -\tfrac{1}{2\e}|x'|^2+ p'\cdot x' + p_n x_n \text{ in } B_r\cap \{x_n\geq 0\}.
\]
\item There exists $p=(p',p_n)\in \R^n$ such that the following holds: For every $\eta>0$ there is a radius $r\in(0,1)$ such that
\[
\begin{cases}
u - (-\frac{1}{2\e}|x'|^2+ p\cdot x) \geq -\eta x_n \text{ in } B_r\cap \{x_n\geq 0\},\\
u - (-\frac{1}{2\e}|x'|^2+ p\cdot x) \leq \eta x_n \text{ in } B_r\cap \{x_n>\eta|x|\}.
\end{cases}
\]
\end{enumerate}
\end{lemma}

\begin{proof}
Let $\phi \in C(\overline{B_1\cap\{x_n>0\}})$ be defined by the boundary value problem
\[
\begin{cases}
\cM^+_{\l,\L}(D^2\phi) + \L|D\phi| + \L \phi_+ = -\L \text{ in } B_1\cap \{x_n>0\},\\
\phi=-\tfrac{1}{2\e}|x'|^2+p'\cdot x' \text{ on } B_1\cap\{x_n=0\},\\
\phi=\min_{\p B_1\cap\{x_n>0\}} u \text{ on } \p B_1\cap\{x_n>0\}.
\end{cases}
\]
Global estimates for convex operators imply that this problem has a classical solution with $\phi \in C^1(B_{1/2}\cap\{x_n\geq 0\})$, \cite{MR661144}.

This barrier shows that for $r\in (0,1/2)$ the following construction is well defined as a finite number
\[
\r(r) := \sup\{\r\in \R \ | \ u(x)\geq -\tfrac{1}{2\e}|x'|^2 + p'\cdot x' + \r x_n \text{ for all } x\in B_r\cap \{x_n>0\}\}.
\]
The first alternative in the theorem holds if $\lim_{r\to0^+}\r(r) = +\8$. So let us assume that $p_n := \lim_{r\to0^+}\r(r) < +\8$.

By construction we automatically have the lower bound and all we need to show is the upper bound in non-tangential domains of the form $B_r\cap \{x_n>\eta|x|\}$. Assume by contradiction that there is some $\eta\in (0,1/2)$ and a sequence $\{y_k\}$ converging to the origin with $(y_k)_n> \eta|y_k|$ such that
\begin{align}\label{eq:har_hyp}
u(y_k) > -\tfrac{1}{2\e}|y_k'|^2 + p\cdot y_k+\eta(y_k)_n.
\end{align}

Let $r_k:= |y_k| \in (0,1/4)$ and consider the following construction defined in $B_2\cap \{x_n>0\}$
\[
v_k(x) := r_k^{-1}u_k(r_kx) - (-\tfrac{r_k}{2\e} |x'|^2 + p'\cdot x'+\r(2r_k)x_n).
\]
In other words, $v_k$ is a Lipschitz rescaling of the non-negative difference between $u_k$ and the approximating polynomial over $B_{2r_k}\cap \{x_n>0\}$
\[
u(x) = -\tfrac{1}{2\e} |x'|^2 + p'\cdot x'+\r(2r_k)x_n + r_kv_k(r_k^{-1}x).
\]

We will use that the function $v_k$ also satisfies in the viscosity sense
\[
\begin{cases}
\cM^+_{\l,\L}(D^2v) + \L r_k|Dv| + \L r_k^2 v_+ + Cr_k \geq 0 \text{ in } B_2\cap\{x_n>0\},\\
\cM^-_{\l,\L}(D^2v) - \L r_k|Dv| - \L r_k^2 v_- - Cr_k \leq 0 \text{ in } B_2\cap\{x_n>0\},
\end{cases}
\]
for some constant $C$ independent of $k$. The verification of this equation in the viscosity sense is straightforward from the equation for $u$.

Given \eqref{eq:har_hyp}, we get that for $v_k$ the following must hold
\[
\sup_K v_k \geq \eta^2,
\]
where $K:= \p B_1\cap \{x_n\geq \eta|x|\}$ is a compact subset of $B_2\cap\{x_n>0\}$.

Let $\m:=c_1\l/\L$ for some $c_1>0$ small to be fixed as a universal constant. Thanks to the Harnack inequality we get that
\[
\inf_{B_{1/2}^{n-1}\times\{x_n=\m\}} v_k \geq c-Cr_k,
\]
where the constant $c=c(\m)>0$ depends on $\mu$. By taking $k$ sufficiently large we get that $\inf_{B_{1/2}^{n-1}\times\{x_n=\m\}} v_k \geq c_0>0$, for $c_0=c(\m)/2$ independent of $k$. The goal is to propagate this lower bound towards $\{x_n=0\}$ with the use of a lower barrier for $w_k(x) := v_k(x)/x_n$. This quotient satisfies the following degenerate elliptic equation in the viscosity sense\footnote{For $v,w\in\R^n$, we denote $v\odot w:= \tfrac{1}{2}(v\otimes w + w\otimes v)\in \R^{n\times n}_{\rm sym}$}
\[
\cM^+_{\l,\L}(x_n D^2w+2e_n\odot Dw) + \L r_k|x_nDw+we_n| + \L r_k^2x_nw + Cr_k \geq 0 \text{ in } B_1^{n-1}\times[0,1).
\]

Let
\[
\phi(x) := \mu + x_n - 8\mu|x'|^2.
\]
Then $\phi\leq 0$ on $(\p B_{1/2}^{n-1}\times(0,\mu]) \cup (B_{1/2}^{n-1}\times\{x_n=0\})$ and
\begin{align*}
\cM^+_{\l,\L}(x_nD^2\phi+2e_n\odot D\phi)(x) &= \sup_{\l I\leq A\leq \L I} \1-x_n\sum_{i,j=1}^{n} a_{ij}\p_{ij}\phi(x)-2\sum_{i=1}^n a_{in}\p_i\phi(x)\2,\\
&= \sup_{\l I\leq A\leq \L I} \116\mu\sum_{i=1}^{n-1} (a_{ii}x_n+2a_{in}x_i)-2a_{nn}\2,\\
&\leq C\m\L-2\l,\\
&\leq -\l.
\end{align*}
The last inequality can be achieved by taking $c_1= \m\L/\l$ sufficiently small.

The remaining terms can be considered lower order perturbations
\[
r_k|x_nD\phi(x)+\phi(x)e_n| + r^2_k x_n \phi(x) \leq C\mu \leq \l/(2\L).
\]
Once again we require $c_1$ to be perhaps even smaller than before for the last inequality to hold.

Hence, for $k\gg 1$
\[
\cM^+_{\l,\L}(x_n D^2\phi+2e_n\odot D\phi) + \L r_k|x_nD\phi+\phi e_n| + \L r^2_k x_n\phi \leq -\l/2\leq -Cr_k \text{ in } B_{1/2}^{n-1}\times(0,\m].
\]
By comparison, $w_k\geq c_0\phi/(2\mu^2)$ in $B_{1/2}^{n-1}\times(0,\m]$.

We finally conclude by noticing that, for $\m$ sufficiently small, $\phi$ is bounded away from zero in the cylinder $B_\m^{n-1}\times(0,\m]$, hence for some $\theta>0$ independent of $k$
\[
\inf_{B_\m^{n-1}\times(0,\m]} v_k/x_n \geq \theta.
\]

Finally, when we transfer this information back to $u$ we get that
\[
u(x) \geq -\tfrac{1}{2\e}|x'|^2 + p'\cdot x' + (\r(2r_k)+\theta)x_n \text{ in } B^{n-1}_{\mu r_k}\times[0,\mu r_k).
\]
This means that $\r(\mu r_k)-\r(2r_k) \geq \theta>0$, which contradicts the convergence of $\r(r)$ to a finite limit as $r\to0^+$.
\end{proof}

In the following corollary we assume that the trace of the solution over the boundary gets trapped between two paraboloids. The conclusion is a bit stronger than differenciability.

\begin{corollary}\label{cor1}
Let $[\l,\L]\ss(0,\8)$, $\e\in(0,1)$, $p'\in \R^{n-1}$ with $|p'|\leq \L$, and let $u\in C(\overline{B_1\cap\{x_n> 0\}})$ satisfy
\[
\begin{cases}
\cM^+_{\l,\L}(D^2u) + \L|Du| + \L u_+ \geq -\L \text{ in } B_1\cap\{x_n>0\},\\
\cM^-_{\l,\L}(D^2u) - \L|Du| - \L u_- \leq \L \text{ in } B_1\cap\{x_n>0\},\\
-\tfrac{1}{2\e}|x'|^2+ p'\cdot x' \leq u \leq \tfrac{1}{2\e}|x'|^2+ p'\cdot x' \text{ on } B_1\cap\{x_n=0\}.
\end{cases}
\]
Then $u$ is differentiable at the origin with $Du(0)=(p',p_n)$ and
\[
\liminf_{\{x_n>0\}\ni x \to 0} \frac{u(x) - (-\tfrac{1}{2\e}|x'|^2+ Du(0)\cdot x)}{x_n} = \limsup_{\{x_n>0\}\ni x \to 0} \frac{u(x) - (\tfrac{1}{2\e}|x'|^2+ Du(0)\cdot x)}{x_n} =0.
\]
\end{corollary}

Notice that we could also state the limits in the conclusion by saying that
\begin{align*}
u(x) \geq Du(0)\cdot x - \tfrac{1}{2\e}|x'|^2 + o(x_n),\\
u(x) \leq Du(0) \cdot x + \tfrac{1}{2\e}|x'|^2 + o(x_n).
\end{align*}
This is a finer estimate than the differentiablity of $u$ at zero given by $u(x) = Du(0)\cdot x + o(|x|)$.

\subsubsection{Eikonal regime}\label{sec:eik_reg}

In this section we present some regularity estimates for first order equations with quasi-convex operators $H\in C(\R^n)$. We measure this in terms of the modulus of continuity given by the one-homogeneous functions
\[
\phi^\pm(x) :=  \pm \max_{p\in \{H\leq0\}} (\pm p\cdot x).
\]

The first few results (up to Corollary \ref{cor:phi_m_sub}) are known in the literature, however the hypotheses seem to be a bit different. In \cite{MR1080619} and also \cite[Section 2]{MR4328923} the operators are assumed to be convex. We decided to include complete proofs in the quasi-convex setting for the results we will be needing ahead.

\begin{lemma}\label{lem:H_below}
Let $H\in C(\R^n)$ be a first-order operator with bounded sub-level sets and let $u \in C(B_r)$ be a viscosity sub-solution of $Hu\leq 0$ in $B_r$. Then $u \leq u(0) + \phi^+$ in $B_r$.
\end{lemma}

\begin{proof}
Given $R\gg 1$, consider
\[
\phi_R(x) := \max_{p\in Z_R} p\cdot x \qquad\text{ where }\qquad Z_R := \bigcap_{\overline{B_R(p)}\supseteq \{H\leq0\}} \overline{B_R(p)}.
\]
Notice then that $\phi_R\searrow \phi^+$ as $R\to\8$. Indeed, the inequality $\phi_R\geq \phi^+$ and the monotonicity of the sequence follow because $Z_R$ decreases to $\operatorname{conv}(\{H\leq0\})\supseteq \{H\leq 0\}$, the convex hull of $\{H\leq0\}$. Let us fix $x^* \in \R^n$ and let $p_R\in Z_R$ such that $\phi_R(x^*) = p_R\cdot x^*$. By compactness we get that for some sequence $R_i\to\8$ the points $p_{R_i} \to p^* \in \operatorname{conv}(\{H\leq0\})$. Given that $p^* = \l p_1+(1-\l)p_2$ for some $p_1,p_2\in \{H\leq0\}$ and $\l\in[0,1]$, we get that
\[
\lim_{R\to\8}\phi_R(x^*) \geq \phi^+(x^*) = \l\phi^+(x^*) + (1-\l)\phi^+(x^*) \geq (\l p_1+(1-\l)p_2)\cdot x^* = \lim_{R\to\8}\phi_R(x^*).
\]

The advantage of this construction is that $Z_R$ is strictly convex and then $\phi_R$ is differentiable outside the origin. By noticing that $\phi_R$ is also $1$-homogeneous we get the following identity for any $x\in\R^n\sm\{0\}$,
\[
\phi_R(x) = D\phi_R(x) \cdot x.
\]

To check that $Z_R$ is strictly convex we consider $p_1\neq p_2\in Z_R$, $p_{1/2}:= (p_1+p_2)/2$, and show that for $r = R-\sqrt{R^2-|p_2-p_1|^2/4}$ it happens that $B_r(p_{1/2})\ss Z_R$. For any $\overline{B_R(p_0)}\supseteq \{p_1,p_2\}$ we clearly have that $B_r(p_{1/2})\ss \overline{B_R(p_0)}$. As this happens for any $\overline{B_R(p_0)}\supseteq \{H\leq0\}$ we immediately conclude that $B_r(p_{1/2})\ss Z_R$.

To see that the convex function $\phi_R$ is differentiable for any $x^*\neq 0$ we notice that $\phi_R(x^*) = p^*\cdot x^*$ for a unique $p^*\in Z_R$. Otherwise, $\phi_R(x^*) = p_1\cdot x^* = p_2\cdot x^* = p_{1/2}\cdot x^*$ for $p_1\neq p_2\in Z_R$ and $p_{1/2}= (p_1+p_2)/2$. Then some $p \in B_r(p_{1/2})\ss Z_R$ would improve the maximum of $p\cdot x^*$, contradicting the definition of $\phi_R(x^*)$.

To finally show that $u(x)\leq u(0)+\phi_R(x)$ for $R\gg1$, assume by contradiction that
\[
m:= \sup_{x\in B_r} (u(x)-u(0) - \phi_R(x)) >0.
\]
Hence, for some $\eta>0$ the following function must touch $u$ from above at some $x^*\in B_r\sm\{0\}$
\[
\psi(x) := u(0)+\tfrac{m}{2}+\phi_R(x)+\tfrac{\eta}{r-|x|}.
\]

At this point we must notice that for any $t>0$, $H(D\phi_R(x^*)+t x^*)>0$. Indeed, if instead $H(D\phi_R(x^*)+t x^*)\leq 0$ we get the contradiction
\[
\phi_R(x^*) = \max_{p\in Z_R} p\cdot x^* \geq (D\phi_R(x^*)+t x^*)\cdot x^*>D\phi_R(x^*)\cdot x^* = \phi_R(x^*).
\]
This observation implies that $H(D\psi(x^*))$ is strictly positive, which is inconsistent with $u$ being a sub-solution of the equation.
\end{proof}

By invoking a covering argument and reversing the roles of the points in the previous lemma we obtain the following corollary.

\begin{corollary}\label{cor:H_below}
Let $\W\ss\R^n$ be open convex set. Let $H\in C(\R^n)$ be a first-order operator with bounded sub-level sets and let $u \in C(\W)$ be a sub-solution of $Hu\leq 0$ in $\W$. Then for any $x,y \in B_r(x_0) \ss B_{2r}(x_0) \ss \W$
\[
\phi^-(x-y)\leq u(x) -u(y) \leq \phi^+(x-y).
\]
\end{corollary}

By Corollary \ref{cor:perron}, we deduce that $\phi^+$ is a sub-solution of $H\phi\leq 0$ in $\R^n$. Additionally, thanks to the following lemma, we can conclude that $\phi^-$ is as well a sub-solution of $H\phi\leq 0$ in $\R^n$ if $\{H\leq0\}$ is convex (or if we assume that $H$ is quasi-convex). It is worth recalling that in the previous proof, we have shown that $\phi^\pm$ are classical solutions of $H\phi=0$ in $\R^n\sm \{0\}$ if $\{H\leq0\}$ is strictly convex.

\begin{lemma}\label{lem:inf_conv}
Let $H\in C(\R^n)$ be a first-order quasi-convex operator with bounded sub-level sets. Let $\W\ss\R^n$ be open, $\{u_k\} \ss C^1(\W)$ be an equicontinuous sequence of functions satisfying $Hu \leq 0$ in $\W$, and assume that $u(x) := \inf u_k(x)$  finite valued. Then $u$ is a viscosity sub-solution of $Hu \leq 0$ in $\W$.
\end{lemma}

This result is also valid for Lipschitz solutions and appears for instance in \cite[Corollary 2.36 and Exercise 30]{MR4328923} as a consequence of the characterization of viscosity solutions for convex Hamiltonians by Barron and Jensen in \cite{MR1080619}.

\begin{proof}
The function $u$ as a uniform limit of functions the form $v_k := \min(u_1,\ldots,u_k)$. Hence the result will follow by the stability of viscosity sub-solutions once we show that each $v_k$ is a sub-solution of $Hu=0$ in $\W$. By an inductive argument we can further reduce the analysis to the case $k=2$.

Let $\varphi \in C^1(B_r(x_0))$ be a test function touching $v_2$ strictly from above at $x_0$ with $B_r(x_0)\ss\W$. If $u_1(x_0)<u_2(x_0)$ we get that $\varphi$ touches $u_1$ from above at $x_0$ over a neighborhood of $x_0$, so we conclude using that $u_1$ is a viscosity sub-solution. The case  $u_2(x_0)<u_1(x_0)$ can be treated in the same way and then we are left with the alternative $u_1(x_0)=u_2(x_0)$.

If $p:= Du_1(x_0)\neq q:= Du_2(x_0)$, then ${u_1=u_2}$ forms a $C^1$ surface around $x_0$ with normal given by $Du_1-Du_2\neq 0$. We can then choose a system of coordinates centered at $x_0$ such that $p'= q'$ and $p_n < q_n$. Consequently, the exterior normal of $\{u_1>u_2\}$ at the origin is the vector $e_n$.

As $\varphi$ touches $\min(u_1,u_2)=u_1=u_2$ over $\{u_1=u_2\}$ from above at $x_0$, we have that the tangential derivatives must coincide
\[
(D\varphi(x_0))' = p'= q'.
\]
On the other hand, for the normal derivatives we must also have that $p_n\leq \p_n\varphi(x_0) \leq q_n$. For instance, $p_n\leq \p_n\varphi(x_0)$ follows because $\varphi -\varphi(x_0)\geq u_1-u_1(x_0)$ over $\{u_1<u_2\}$, which has $e_n$ as a interior normal vector at $x_0$.

Now we notice that for $\l:= (\p_n\varphi(x_0)-p_n)/(q_n-p_n)\in[0,1]$ we get that by the convexity of $\{H\leq0\}$ and using that $u_1$ and $u_2$ are classical sub-solutions
\[
D\varphi(x_0) = \l Du_1(x_0)+(1-\l)Du_2(x_0) \in \{H\leq 0\}
.
\]

If instead, $Du_1(x_0)= Du_2(x_0)$ we must also have $D\varphi(x_0)=Du_1(x_0)= Du_2(x_0) \in \{H\leq0\}$. Otherwise, we have that $\{\varphi<u_1\}$ and $\{\varphi<u_2\}$ have the same exterior normal vector at $x_0$, so that $\{\varphi\geq u_1\}\cup\{\varphi\geq u_2\}$ can not cover a neighborhood of $x_0$. This contradicts $\varphi\geq \min(u_1,u_2)$ around $x_0$.
\end{proof}

\begin{corollary}\label{cor:phi_m_sub}
Let $H\in C(\R^n)$ be a first-order quasi-convex operator with bounded sub-level sets. Then $\phi^-$ is a viscosity solution of the equation $H\phi = 0$ in $\R^n$.
\end{corollary}

The following lemmas are the boundary estimates we will use for the eikonal problem. The first result says that under certain conditions a paraboloid $\varphi$ over $B_1\cap \{x_n=0\}$ can be extended to a paraboloid with $H(D\varphi(0))\leq 0$.

\begin{lemma}\label{lem:H_cons}
Let $H\in C(\R^n)$ and let $u\in C^{0,1}(B_1^{n-1}\times (-1,0])$ be a viscosity sub-solution of $Hu\leq 0$ in $B_1^{n-1}\times (-1,0)$ such that for some $p'\in \R^{n-1}$ and $\e>0$
\[
u \leq u(0) + \tfrac{1}{2\e}|x'|^2 + p'\cdot x' \text{ on } B_1\cap \{x_n=0\}.
\]
Then there exists $p=(p',p_n)\in \R^n$ such that $H(p)\leq 0$.
\end{lemma}

\begin{proof}
For $L:= [u]_{C^{0,1}(B_1\cap\{x_n\leq 0\})}$ and $r\in (0,1/2)$ small, we get that
\[
u \leq u(0)+\tfrac{r^2}{2\e} |p'|r+\sqrt 2 Lr \text{ in } B_r^{n-1}\times [-r,0].
\]
Then, by letting $\r:=r/\e + |p'|+\sqrt 2 L$, we obtain that
\[
\varphi(x) := u(0) + \tfrac{1}{\e}|x'|^2 + p'\cdot x'-\r x_n - \tfrac{r^2}{4\e} > u \text{ on } (\p B_r^{n-1}\times [-r,0]) \cup (B_r^{n-1}\times \{-r\}).
\]

Given that $\varphi(0)< u(0)$ we must have that for some $\eta>0$ the function $\psi = \varphi - \eta x_n^{-1}$ touches $u$ from above at some $x_r\in B_r^{n-1}\times (-r,0)$. At such contact point we have
\[
p_r := D\psi(x_r) = (\tfrac{2}{\e}x_r' + p', -\r + \eta(x_r)_n^{-2}) \in B_L
\]
and $H(\tfrac{2}{\e}x_r' + p', -\r + \eta(x_r)_n^{-2})\leq 0$.

Then $|p_r'-p'|\leq 2r/\e\to 0$ as $r\to 0$ and by compactness there is an accumulation value such that $(p_r)_n\to p_n$ for some sequence. We can now conclude the proof thanks to the continuity of $H$.
\end{proof}

To finish this section let us consider $\eta>0$ and
\[
\phi_1^+(x) := \max_{p\in \{H+\eta\leq 0\}} p\cdot x.
\]
For any sub-solution $u$ of $Hu+\eta\leq 0$ in $B_1\cap \{x_n<0\}$ with boundary data $u\leq \varphi$ over $B_1\cap \{x_n=0\}$, the lower envelope given by the Hopf-Lax formula
\[
\psi(x):= \inf_{y'\in \R^{n-1}} \varphi(y')+\phi_1^+(x-y')
\]
is an upper bound for $u$. This is a consequence of Corollary \ref{cor:H_below}.

The following lemma gives us a further bound from above for $\psi$ whenever $\varphi$ can be extended to a paraboloid with $H(D\varphi(0))=0$.

\begin{lemma}\label{lem:hopf}
Let $H\in C(\R^n)$ be a first-order quasi-convex operator with bounded sub-level sets. Let $\eta>0$, and assume that there exists $p_1 = (p_*',\r_1) \in \R^n$ such that $H(p_1)+\eta\leq 0$. Let $p_0 = (p_*',\r_0) \in \R^n$ such that
\[
\r_0 = \min\{\r\in \R \ | \ H((p_*',\r)) \leq 0\}<\r_1.
\]
Let $\e>0$ and
\[
\varphi(x) := \tfrac{1}{2\e}|x'|^2 + p_0\cdot x, \qquad \psi(x):= \inf_{y'\in \R^{n-1}} \varphi(y')+\phi_1^+(x-y').
\]
Then for some $r>0$ sufficiently small $\psi\leq \varphi$ in $B_r\cap\{x_n\leq0\}$.
\end{lemma}

\begin{proof}
The goal is to see that we can take $r$ sufficiently small such that for any $x \in B_r\cap\{x_n\leq0\}$ there exists some $y'\in B_1\cap \{x_n=0\}$ for which
\[
\frac{1}{2\e}|y'|^2 +p_*'\cdot y' + \max_{p\in\{H+\eta\leq 0\}}p\cdot (x-y')\leq \frac{1}{2\e}|x'|^2 + p_0\cdot x.
\]

Let $r_0\in(0,1)$ such that $B_{2r_0}(p_0) \cap \{H+\eta\leq 0\} = \emptyset$. By Hahn-Banach there exists some $z \in \p B_1$ such that
\[
\min_{p \in \{H+\eta\leq 0\}} z\cdot p \geq \max_{q \in B_{r_0}(p_0)} z \cdot q + r_0.
\]
Notice that $z_n>0$. Moreover, we can even give a strictly positive lower bound for $z_n$. Using that $p_0 \in B_{r_0}(p_0)$ and $p_1\in \{H+\eta\leq 0\}$ are separated by the previous application of Hanh-Banach we obtain that
\begin{align*}
&z\cdot p_1 = z'\cdot p_*'+z_n \r_{1} \geq z\cdot p_{0}+r_0 = z'\cdot p_*'+z_n\r_{0}+r_0,\\
\Rightarrow \qquad &z_n \geq r_0/(\r_1-\r_{0})>0.
\end{align*}

For any $\l \in (0,\e r_0)$ the vector $\bar z=\l z$ satisfies
\[
\min_{p \in \{H+\eta\leq 0\}} \bar z\cdot p \geq \max_{q \in B_{r_0}(p_{0})} \bar z \cdot q + \tfrac{1}{\e}\bar z_n^2.
\]

If we take $r$ smaller than $\e r_0\min(r_0/(\r_1-\r_{0}),1/2)$ and $x\in B_r\cap\{x_n\leq 0\}$, we get that $\l := -x_n/z_n \in (0,\e r_0)$ so the previous estimate holds. Let $y' = x + \bar z = x - (x_n/z_n)z \in B_1\cap\{x_n=0\}$ so that $|y'+x|=|\bar z+2x|\leq |\bar z|+2|x| < 2\e r_0$ and then
\begin{align*}
\min_{p \in \{H+\eta\leq 0\}} (y'-x)\cdot p &\geq \max_{q \in B_{r_0}(p_{0})}  (y'-x) \cdot q+\tfrac{1}{\e}x_n^2,\\
&\geq (y'-x)\cdot(\tfrac{1}{2\e}(y'+x) +p_{0}) +\tfrac{1}{2\e}x_n^2,\\
&= \tfrac{1}{2\e}|y'|^2-\tfrac{1}{2\e}|x'|^2+p_{0}\cdot (y'-x).
\end{align*}
This is the desired inequality which concludes the proof.
\end{proof}

\subsection{Proof of the comparison principle}

As a first step towards the proof of Theorem \ref{thm:comp} we show that the comparison principle holds under the following assumptions for the solutions $u, v \in C(B_1^{n-1}\times(-1,1))\cap C^{0,1}(B_1^{n-1}\times(-1,0])$ with respect to $\e\in(0,1)$, $[\l,\L]\ss(0,\8)$, $H \in C(\R^n)$, and $\eta>0$:
\begin{itemize}
\item[\namedlabel{hyp:con}{(C')}] The following functions are convex
\[
x' \in B_1^{n-1} \mapsto u(x',0)+\tfrac{1}{2\e}|x'|^2 \quad \text{ and } \quad x' \in B_1^{n-1} \mapsto -v(x',0)+\tfrac{1}{2\e}|x'|^2.
\]
\item[\namedlabel{hyp:ue2}{(UE')}] Both of the functions $u$ and $v$ satisfy in the viscosity sense
\[
\begin{cases}
\cM^+_{\l,\L}(D^2u) + \L|Du| + \L u_+ + \L \geq 0 \text{ in }  B_1^{n-1}\times(0,1),\\
\cM^-_{\l,\L}(D^2u) - \L|Du| - \L u_- - \L \leq 0 \text{ in }  B_1^{n-1}\times(0,1).
\end{cases}
\]
\item[\namedlabel{hyp:sub2}{(Sub')}] The function $u$ satisfies in the viscosity sense
\[
Hu +\eta \leq 0 \text{ in } B_1^{n-1}\times(-1,0].
\]
\item[\namedlabel{hyp:sup2}{(Sup')}] The function $v$ satisfies in the viscosity sense
\[
Hv \geq 0 \text{ in } B_1^{n-1}\times(-1,0].
\]
\end{itemize}

\begin{lemma}\label{main:lem}
Let $\eta>0$ and let $H \in C(\R^n)$ a first-order quasi-convex operator with bounded sub-level sets. Let $\e\in(0,1)$, $[\l,\L]\ss(0,\8)$, and let $u, v \in C(B_1^{n-1}\times(-1,1))\cap C^{0,1}(B_1^{n-1}\times(-1,0])$ satisfy \ref{hyp:con}, \ref{hyp:ue2}, \ref{hyp:sub2}, \ref{hyp:sup2} above. Then $u$ can not touch $v$ from below at the origin.
\end{lemma}

In the following proof we will use once again the notation $u_\pm$ for the restriction of a given function to $\W_\pm\cup\G$.

\begin{proof}
Assume by contradiction that $u\leq v$ with $u(0)=v(0)=0$ (without loss of generality). By the convexity hypothesis on $u$ and $v$ we get that for some $p_*'\in \R^{n-1}$
\begin{align}\label{eq2}
-\tfrac{1}{2\e}|x'|^2 + p_*'\cdot x' \leq u(x',0) \leq v(x',0) \leq \tfrac{1}{2\e}|x'|^2 + p_*'\cdot x' \text{ in } B_1^{n-1}.
\end{align}

Given $t\in[0,1]$ consider
\begin{align*}
\r_t &:= \min\{\r \in \R \ | \ H(p_*',\r)+ \eta t\leq 0\},\\
\phi^\pm_t(x) &:= \pm\max_{p\in \{H+\eta t\leq 0\}}(\pm p\cdot x),\\
\varphi^\pm_t(x) &:= \pm \tfrac{1}{2\e}|x'|^2 + p_*'\cdot x' + \r_t x_n.
\end{align*}

Thanks to the Lemma \ref{lem:H_cons} we known that $\r_t$ are well defined as finite numbers, and by the continuity of $H$
\[
0\leq t_1<t_2\leq 1 \qquad\Rightarrow\qquad\r_{t_1}<\r_{t_2}\qquad\Rightarrow\qquad \begin{cases}
    \varphi_{t_1}^\pm > \varphi_{t_2}^\pm \text{ in } \{x_n<0\},\\
    \varphi_{t_1}^\pm < \varphi_{t_2}^\pm \text{ in } \{x_n>0\}.
\end{cases}
\]

The functions $\phi^\pm_t$ also have a monotonicity property
\[
0\leq t_1<t_2\leq 1 \qquad\Rightarrow\qquad \begin{cases}
    \phi^-_{t_1}<\phi^-_{t_2} \text{ in } \R^n\sm\{0\},\\
    \phi^+_{t_1}>\phi^+_{t_2} \text{ in } \R^n\sm\{0\}.
\end{cases} .
\]

By Corollary \ref{cor:phi_m_sub} we have that $\phi^-_t$ satisfies in the viscosity sense,
\begin{align}\label{eq:phi_m_eq}
H(D\phi_t^-) + \eta t = 0 \text{ in }\R^n.
\end{align}
Meanwhile, $\phi^+_t$ satisfies instead
\[
H(D\phi_t^+) + \eta t \leq 0 \text{ in }\R^n.
\]

By Corollary \ref{cor:H_below} we obtain that for every $x,y\in B_1\cap\{x_n\leq0\}$
\begin{align}\label{eq:H_below}
\phi^-_1(x-y) \leq u(x) - u(y) \leq \phi^+_1(x-y).
\end{align}

Finally we also point out that $\varphi_t^\pm(x) = -\tfrac{1}{2\e}|x'|^2 + p_*'\cdot x'$ on $\{x_n=0\}$  such that
\begin{align}\label{eq:varphi_main}
\varphi_t^-\leq u\leq v\leq \varphi_t^+ \text{ in }B_1^{n-1}
\end{align}
and at the origin we get that $H(D\varphi_t^\pm(0)) =- \eta t$.

\textbf{1.} Our first goal is to show that for some radius $r \in(0,1)$ sufficiently small,
\begin{align}\label{eq3}
\varphi^-_{1/2} \overset{?}{\leq} v \text{ in } B_r\cap \{x_n\leq 0\}.
\end{align}

Let $r\in(0,1)$ be a sufficiently small radius such that $H(D\varphi^-_{1/2}) \in (-\eta,0)$ in $B_{r}$, which is possible because of continuity and $H(D\varphi^-_{1/2})=-\eta/2$. Consider now the function
\[
\psi := \min(\varphi^-_{1/2},\phi^-_0+c), \qquad c:= \inf_{\partial B_r} (\phi^-_1-\phi^-_0)>0.
\]

By construction and the use of \eqref{eq:varphi_main} and \eqref{eq:H_below} we have that $\psi\leq v$ on $\p(B_{r}\cap\{x_n< 0\})$. Over $B_r\cap\{x_n=0\}$ we use that $\psi\leq \varphi^-_{1/2}\leq v$. Over $\p B_r\cap\{x_n<0\}$ we use that $\psi\leq \phi^-_1\leq u\leq v$. Moreover, $\psi$ is also sub-solution of the equation $H\psi\leq 0$ in $B_r^{n-1}\times(-r,0)$. This is a consequence of Lemma \ref{lem:inf_conv}. By comparison, we conclude that $\psi\leq v$ in $B_{r}\cap\{x_n\leq  0\}$.

Finally notice that by continuity $\psi=\varphi_{1/2}^-$ in some neighborhood of the origin. With this we have shown that \eqref{eq3} must hold for some sufficiently small radius.

\textbf{2.} Now we look at the other side, $\{x_n\geq0\}$. The new goal is to see that
\[
\p_n v_+(0)\overset{?}{\leq} \r_{1/2}.
\]

This derivative is well defined thanks to the Lemma \ref{cor1}. If we assume by contradiction that the opposite inequality holds, then Lemma \ref{cor1} and the previous lower bound for $v_-$ would imply that the paraboloid $\varphi^-_{1/2}$ touches $v$ from below at the origin. This contradicts $v$ being a super-solution of $Hv\geq 0$ given that $H(D\varphi^-_{1/2}(0)) = -\eta/2<0$.

\textbf{3.} Let us show now that for some radius $r\in(0,1)$
\begin{align}\label{eq:3mt}
\varphi^+_{3/4} \overset{?}{\geq} u \text{ in } B_r\cap\{x_n\geq0\}.
\end{align}

Lemma \ref{cor1} applies to $u_+$, so we get that $\p_nu_+(0)$ is well defined. From the contact of $v_+$ and $u_+$ at the origin we get that $\p_nu_+(0)\leq \p_nv_+(0) \leq \r_{1/2}$. Moreover, by the same lemma we conclude that for some $r\in(0,1)$ sufficiently small
\[
u_+-\varphi^+_{1/2} \leq (\r_{3/4}-\r_{1/2})x_n \text{ in } B_r\cap\{x_n\geq 0\}.
\]
This inequality is equivalent to \eqref{eq:3mt}.

\textbf{4.} As a final step we will extend the bound in \eqref{eq:3mt} to the other side
\[
\varphi^+_{3/4} \overset{?}{\geq} u \text{ in } B_r.
\]

By \eqref{eq:H_below} we know that
\[
u(x) \leq \psi(x) := \inf_{y'\in B_1\cap \{x_n=0\}}\varphi^+_1 (y') + \phi^+_1(x-y') \text{ for } x\in B_1\cap\{x_n\leq0\}.
\]
By Lemma \ref{lem:hopf} we know that $\psi\leq\varphi^+_{3/4}$ in $B_r$.

This last step concludes the proof as it contradicts $u$ being a sub-solution of $Hu+\eta\leq 0$. The test function $\varphi^+_{3/4}$ satisfies at the contact point that $H(D\varphi_{3/4}^+(0)) = -3\eta/4>-\eta$.
\end{proof}

\subsubsection{Proof of Theorem \ref{thm:comp}}

Assume by contradiction that $u$ and $v$ satisfy the hypotheses of the theorem, however $u\leq v$ on $\p\W$ and $m := \sup_\W (u-v)>0$.

Recall the constructions from Section \ref{sec:inf_sup} and consider the sup and inf-convolutions (in the directions parallel to $\{x_n=0\}$) of $u$ and $v$ respectively and with respect to some parameter $\e>0$. Let $M= \max(\|u\|_{L^\8(\overline\W)},\|v\|_{L^\8(\overline\W)})$ and $r=2\sqrt{M\e}$. By the uniform convergence convergence of $u^\e\searrow u$ and $v_\e\nearrow v$, we have that for $\e$ sufficiently small we can enforce $\max_{\p\W^r} (u^\e-v_\e) \leq m/4$ meanwhile $\sup_{\W^r} (u^\e-v_\e)\geq m/2$.

The hypotheses \ref{hyp:hm} and \ref{hyp:hp}, together with the assumptions $u,v \in C(\overline\W)\cap C^{0,1}(\W_-)$ allows to get that for $\e$ even smaller, $u^\e$ is a sub-solution of
\begin{align*}
\begin{cases}
    H_-u+3\eta/4\leq 0 \text{ in } \W^r_-\cup\G^r,\\
    H_+u+3\eta/4\leq 0 \text{ in } \W^r_+,
\end{cases}
\end{align*}
meanwhile $v_\e$ is a super-solution of
\begin{align*}
\begin{cases}
    H_-v+\eta/4 \geq 0 \text{ in } \W^r_-\cup\G^r,\\
    H_+v+\eta/4\geq 0 \text{ in } \W^r_+.
\end{cases}
\end{align*}

If $\argmax_{\W_r}(u^\e-v_\e)\cap (\W_-^r\cup\W_+^r)\neq \emptyset$, then we get a contradiction by the standard comparison principle. Let us assume then that there is some $x_0 \in \argmax_{\W_r}(u^\e-v_\e) \cap \G^r$ and $c=(u^\e-v_\e)(x_0)>0$. Therefore, the function $u^\e$ touches $v_\e+c$ from below at $x_0 \in \G$ over some small neighborhood $\overline{B_\r(x_0)}\ss\W^r$.

Let $H_0(p) := H_-(p,u(x_0),x_0)$. By the continuity of the solutions and the operators, we get that for $\r>0$ sufficiently small, the function $u^\e$ is a viscosity sub-solution of
\[
H_0u +5\eta/8 \leq 0 \text{ in } B_\r(x_0)\cap\{x_n\leq 0\},
\]
and $v_\e+c$ is a viscosity super-solution of
\[
H_0v + 3\eta/8 \geq 0 \text{ in } B_\r(x_0)\cap\{x_n\leq 0\}.
\]
By taking $H:=H_0+3\eta/8$ we get that $u^\e$ and $v_\e+c$ satisfy the translation invariant equations \ref{hyp:sub2} and \ref{hyp:sup2} respectively for the gap $\eta/4>0$.

As in Lemma \ref{lem:lift}, we consider now $\bar u\in C(\overline{B_\r(x_0)})$ to be the lifting of $u^\e$ constructed from the boundary value problem
\[
\begin{cases}
    H_+\bar u + 3\eta/4 = 0 \text{ in } B_\r(x_0)\cap\{x_n>0\},\\
    \bar u = u^\e \text{ on } \overline{B_\r(x_0)} \sm (B_\r(x_0)\cap\{x_n>0\}).
\end{cases}
\]
In a similar way we let $\bar v\in C(\overline{B_\r(x_0)})$ to be given by
\[
\begin{cases}
    H_+\bar v + \eta/4 = 0 \text{ in } B_\r(x_0)\cap\{x_n>0\},\\
    \bar v = v_\e+c \text{ on } \overline{B_\r(x_0)} \sm (B_\r(x_0)\cap\{x_n>0\}).
\end{cases}
\]
The comparison principle for uniformly elliptic equations guarantee that we still have that $\bar u$ touches $\bar v$ from below at $x_0$. The advantage of this construction is that now both $\bar u$ and $\bar v$ now satisfy \ref{hyp:ue2} in $B_r(x_0)$, meanwhile they keep satisfying equations \ref{hyp:sub2} and \ref{hyp:sup2}, by the same argument towards the end in the proof of Lemma \ref{lem:lift}. 

Now that we have all the hypotheses from Lemma \ref{main:lem}, we we have the desired contradiction as $\bar u$ is not allowed to touch $\bar v$ from below at $x_0\in\Gamma$.\qed

\bibliographystyle{plain}
\bibliography{mybibliography}

\end{document}